\documentclass[onefignum,onetabnum]{siamonline220329}

\usepackage[english]{babel}
\usepackage[utf8x]{inputenc}
\usepackage{graphicx}
\usepackage{float}
\usepackage{amsmath,amssymb}
\usepackage{enumitem}
\usepackage{tabularx}
\usepackage{multirow}
\usepackage{multicol}
\usepackage{subcaption}
\usepackage{bm}
\usepackage{comment}
\usepackage{pifont}

\usepackage{tikz}
\usetikzlibrary{decorations.pathreplacing}

\usepackage{ytableau}

\DeclareMathOperator{\I}{\mathcal{I}}

\DeclareMathOperator{\Q}{\mathbb{Q}}
\DeclareMathOperator{\Res}{Res}
\DeclareMathOperator{\wronsk}{Wr}
\DeclareMathOperator{\C}{\mathbb{C}}
\DeclareMathOperator{\Z}{\mathbb{Z}}
\DeclareMathOperator{\Znn}{\mathbb{Z}_{\geqslant 0}}
\newcommand{\Znncomp}[1]{\overline{\mathbb{Z}}_{\geqslant 0}^{#1}}
\newcommand*\rot{\rotatebox{90}}
\newcommand*\OK{\ding{51}}
\newcommand*\NOK{\ding{55}}

\newtheorem{example}{Example}
\newtheorem{notation}{Notation}
\newtheorem{remark}{Remark}

\title{Differential elimination for dynamical models via projections\\ with applications to structural identifiability}

\author{Ruiwen Dong\footnote{\'Ecole Polytechnique, Route de Saclay, 91120 Palaiseau, France; Department of Computer Science, University of Oxford, Oxford, OX1 3QD, United Kingdom}, Christian Goodbrake\footnote{Oden Institute for Computational Engineering and Sciences, University of Texas, Austin, 201 E 24th St, Austin, TX 78712; Mathematical Institute, University of Oxford, Oxford, OX2 6GG, United Kingdom}, Heather A Harrington\footnote{Mathematical Institute, University of Oxford, Oxford, OX2 6GG, United Kingdom; Wellcome Centre for Human Genetics, University of Oxford, Oxford, OX3 7BN, United Kingdom}, and Gleb Pogudin\footnote{LIX, CNRS, \'Ecole Polytechnique, Institute Polytechnique de Paris, 1 rue Honor\'e d'Estienne d'Orves, 91120, Palaiseau, France, email: \url{gleb.pogudin@polytechnique.edu}}}

\date{}

\begin{document}

\maketitle

\begin{abstract}
Elimination of unknowns in a system of differential equations is often required when analysing (possibly nonlinear) dynamical systems models, where only a subset of variables are observable. One such analysis, identifiability, often relies on computing input-output relations via differential algebraic elimination. Determining identifiability, a natural prerequisite for meaningful parameter estimation, is often prohibitively expensive for medium to large systems due to the computationally expensive task of elimination.
We propose an algorithm that computes a description of the set of differential-algebraic relations between the input and output variables of a dynamical system model.
The resulting algorithm outperforms general-purpose software for differential elimination on a set of benchmark models from literature.
We use the designed elimination algorithm to build a new randomized algorithm for assessing structural identifiability of a parameter in a parametric model.
A parameter is said to be identifiable if its value can be uniquely determined from input-output data assuming the absence of noise and sufficiently exciting inputs.
Our new algorithm allows the identification of models that could not be tackled before.
Our implementation is publicly available as a Julia package at~\url{https://github.com/SciML/StructuralIdentifiability.jl}.
\end{abstract}



\section{Introduction}

\subsection{Overview}

Models defined by systems of differential equations are widely used in engineering and sciences.
One of the fundamental challenges in designing, testing, calibrating, and using such models comes from the fact that, in practice, often only few of the variables are observed/measured.
Despite a deluge of data and advances in technology, the appropriate data may be inaccessible (e.g., prohibitively expensive to attach sensors to all parts of a mechanism, infeasible to observe protein complexes in distinct states, or impossible to measure all the proteins in a model in a single experimental set up).
Therefore, the relations involving only observable variables play an important role in the systems and control theory (referred to as \emph{input-output relations}, see, e.g., the textbooks~\cite{Conte2007, Sontag1998} and~\cite{Wang1995}).
For example, computing these input-output relations explicitly is a key step in a so-called differential algebra approach to assessing structural identifiability of a dynamical model~\cite{OllivierPhD, DAISY, LG94, MED2009}.
Other applications include linearization~\cite{Glumineau1996}, model selection~\cite{selection}, parameter estimation~\cite{Verdiere2005, DenisVidal2003}, fault diagnosis~\cite{Jiafan2009, Staroswiecki2001}, and control~\cite{Komatsu2020}.

The problem of computing the input-output relations can be viewed as the \emph{differential elimination problem}: given a system of differential equations
\[
  f_1(\mathbf{x}, \mathbf{y}) = \ldots = f_n(\mathbf{x}, \mathbf{y}) = 0
\]
in two groups of unknowns $\mathbf{x} = (x_1, \ldots, x_s)$ and $\mathbf{y} = (y_1, \ldots, y_\ell)$, describe all equations $g(\mathbf{y}) = 0$ in $\mathbf{y}$ only,
which hold for every solution of the system.
This problem has been one of the central problems in the algebraic theory of differential equations.
Its study has been initiated by Ritt, the founder of differential algebra, in the 1930s~\cite{Ritt}.
He developed the foundations of the characteristic set approach, which has been made fully constructive by Seidenberg~\cite{Seidenberg}.
The algorithmic aspect of this research culminated in the Rosenfeld-Gr\"obner algorithm~\cite{Boulier2, Hubert2003b} implemented in the BLAD library~\cite{blad} (available through {\sc Maple}).
See~\cite{diff_Thomas, WANG2002} for related software.
These algorithms and packages are very versatile: they can be applied to arbitrary systems of polynomials PDEs.
There is a price to pay for such versatility: many interesting examples coming from applications cannot be tackled in a reasonable time.
On the other hand, since differential equations in sciences and engineering are typically used to describe how the system of interest will evolve from a given state, many dynamical models in the literature are described by systems in \emph{the state-space form}:
\begin{equation}\label{eq:state_space}
  \mathbf{x}' = \mathbf{f}(\mathbf{x}, \mathbf{u}), \quad \mathbf{y} = \mathbf{g}(\mathbf{x}, \mathbf{u})
\end{equation}
where $\mathbf{f}$ and $\mathbf{g}$ are tuples of rational functions, $\mathbf{x}$, $\mathbf{y}$, and $\mathbf{u}$ are tuples of differential unknowns (the state, output, and input variables, respectively).
For such a system, one typically wants to eliminate the $x$-variables, that is, compute the \emph{input-output relations}, the relations between the $y$-variables and $u$-variables.
For example, such a computation is sufficient for all the applications mentioned in the beginning.

The contribution of the present paper is two-fold:
\begin{enumerate}
    \item \textbf{Elimination.} We propose \emph{a new way to describe the input-output relations} of~\eqref{eq:state_space} and design \emph{a computationally tractable algorithm} for computing this description. 
    We demonstrate that this algorithm outperforms the existing general-purpose elimination software (e.g., computations that took hours or didn’t finish are computed in minutes, see Table~\ref{tab:comparison_elimination}).

    \item \textbf{Identifiability.} We build \emph{a new randomized algorithm for assessing structural identifiability} of parametric dynamical models on top of the elimination algorithm.
    The algorithm can handle problems that could not be solved by any existing identifiability software (see comparison to the state-of-the-art methods in Table~\ref{tab:comparison_identifiability}).
    Our software is available as a Julia package at~\url{https://github.com/SciML/StructuralIdentifiability.jl}.
\end{enumerate}

The next two subsections describe the contributions in more detail.


\subsection{Elimination} 

We propose to use a projection-based description of the input-output relations of~\eqref{eq:state_space} which can be viewed as a generalization of the state-space form (that is, the form~\eqref{eq:state_space}) itself.
In order to motivate and introduce this description, we adopt the following algebro-geometric viewpoint on~\eqref{eq:state_space}: we consider all the derivatives of the equations in~\eqref{eq:state_space} as an infinite system of equations describing a variety in an infinite-dimensional space with the coordinates $\mathbf{x}, \mathbf{y}, \mathbf{u}, \mathbf{x}', \mathbf{y}', \mathbf{u}', \ldots$.
The points of this variety over $\mathbb{C}$ will be in bijective correspondence (via the Taylor series~\cite[Lemma 3.5]{noether}) with the formal power series solutions of~\eqref{eq:state_space}.
We observe that the variety is rationally parametrized by $\mathbf{x}, \mathbf{u}, \mathbf{u}', \ldots$ (we will refer to these variables as \emph{the base variables}).
Moreover, each equation in~\eqref{eq:state_space} relates these base variables and one non-base variable of the lowest order, that is, one of $\{\mathbf{x}', \mathbf{y}\}$.

We can now generalize such a description to the notion of \emph{a projection-based representation} of the ideal $\mathcal{I}$ generated by the derivatives of~\eqref{eq:state_space} by allowing
\begin{itemize}
    \item any set of base variables such that they form a transcendence basis modulo $\mathcal{I}$ and, if $w^{(i)}_j$ (the $i$-th derivative of $w_j$) with $w \in \{x, y, u\}$ is a base variable, then $w_j^{(0)}, \ldots, w_j^{(i - 1)}$ are base variables as well;
    \item the relations between the base variables and one of the lowest order non-base variables (\emph{projections}) be nonlinear in the latter (i.e., parametrization of the variety may no longer be  rational).
\end{itemize}

\begin{example}\label{ex:toy_example}
A toy example of changing the base variables (base variables are underlined):
\begin{equation}\label{eq:pb-toy}
  \begin{cases}
  x' = \underline{x},\\
  y = \underline{x}^2 + \underline{u}
  \end{cases}
  \;\implies (\text{using }y' = 2xx' + u' = 2x^2 + u') \implies\;\; \begin{cases}
   x^2 = \underline{y} - \underline{u},\\ y' = 2\underline{y} - 2\underline{u} + \underline{u'}.
   \end{cases}
\end{equation}
Note that such a change of the set of base variables may add extraneous prime components to the variety, this subtlety and the way we deal with it are discussed in Remark~\ref{rem:extra}.
\end{example}

Example~\ref{ex:toy_example} also suggests how projection-based representations could be used for differential elimination: observe that the last equation in~\eqref{eq:pb-toy}, $y' = 2y - 2u + u'$, is an input-output equation of minimal order for the original model.
Indeed, we will show that in general, if one considers a set of base variables containing as many derivatives of $y$-variables as possible, a subset of projections will form a projection-based representation of the ideal of input-output relations (see Lemma~\ref{lem:proj_proj}).
The main idea is to replace the $x$-variables in the set of base variables one by one with $y$-variables and their derivatives.
We visualize such a computation for a simple artificial example on Figure~\ref{tab:yound_oscillator}.

\ytableausetup{centertableaux}

\setcounter{table}{0}
\setlength{\tabcolsep}{-1pt}
\begin{table}[H]
\centering
\begin{tabular}{ccccccc}
    $\begin{cases}
      x_1' = x_2,\\
      x_2' = x_3,\\
      x_3' = x_1,\\
      y_1 = x_1 + x_2,\\
      y_2 = x_3
    \end{cases}$ & $\implies$ & $\begin{cases}
      x_1' = x_2,\\
      x_2' = y_2,\\
      x_3 = y_2,\\
      y_1 = x_1 + x_2,\\
      y_2' = x_1
    \end{cases}$ & $\implies$ & $\begin{cases}
      x_1' = y_1 - x_1,\\
      x_2 = y_1 - x_1,\\
      x_3 = y_2,\\
      y_1' = y_1 - x_1 + y_2,\\
      y_2' = x_1
    \end{cases}$ & $\implies$ & $\begin{cases}
      x_1 = y_2',\\
      x_2 = y_1 - y_2',\\
      x_3 = y_2,\\
      y_1' = y_1 - y_2' + y_2,\\
      y_2'' = y_1 - y_2'
    \end{cases}$\\
    & & & & & & \\
    \begin{ytableau}
    \none & \none & \none & \none & \none \\
    x_1 & x_2 & x_3 & \none & \none \\
   \none[x_1] & \none[x_2] & \none[x_3] & \none[y_1] & \none[y_2]
   \end{ytableau}& $\to$ &
   \begin{ytableau}
    \none & \none & \none & \none & \none \\
    x_1 & x_2 & \none & \none & y_2 \\
   \none[x_1] & \none[x_2] & \none[x_3] & \none[y_1] & \none[y_2]
   \end{ytableau}& $\to$ &
    \begin{ytableau}
    \none & \none & \none & \none & \none \\
    x_1 & \none & \none & y_1 & y_2 \\
   \none[x_1] & \none[x_2] & \none[x_3] & \none[y_1] & \none[y_2]
   \end{ytableau}& $\to$ & \begin{ytableau}
    \none & \none & \none & \none & y_2' \\
    \none & \none & \none & y_1 & y_2 \\
   \none[x_1] & \none[x_2] & \none[x_3] & \none[y_1] & \none[y_2]
   \end{ytableau}
\end{tabular}
\captionsetup{name=Figure}
\caption{Elimination via a chain of projection-based representations.
The diagrams in the second row describe the set of base variables at each step.}\label{tab:yound_oscillator}
\end{table}
\setlength{\tabcolsep}{6pt}

\setcounter{table}{0}

The key features of the projection-based representation that allow us to translate such an approach into a practically efficient algorithm are the following:
\begin{itemize}
    \item 
    The projection based representation is given by equations of the projections of the original variety defined by $\mathcal{I}$; hence, \textit{geometric} in nature. In contrast to syntactic representations, such as characteristic sets or Gr\"obner bases, the projection based representation enables us to use tools from constructive algebraic geometry.

    \item The elimination process visualized in Figure~\ref{tab:yound_oscillator} is \emph{adaptive} in the sense that, after each step, the algorithm often has a choice which of the $x$'s to replace with one of the $y$'s (e.g., at the first step in Figure~\ref{tab:yound_oscillator}, we could eliminate any of $x_1, x_2, x_3$ from the base variables).
    By employing simple degree-based heuristics for this choice, we are able to speed up the computation substantially.
    While the described algorithm may be reminiscent of the Gr\"obner walk algorithm~\cite{COLLART1997} and its relatives~\cite{Dahan2008, Golubitsky2009}, especially~\cite{Boulier2010}, the ranking of $y$’s is not predetermined. 
    The distinctive feature of the proposed algorithm is that the final ranking of y’s can be constructed \emph{adaptively and dynamically}.
    Note that the choice of ranking affects only the runtime of the algorithm but not the correctness of the result.
\end{itemize}

Since each step of the algorithm boils down to elimination of a single variable (one of the $x$'s), we use resultants as the main algebraic elimination tool.
While resultants are known to produce extraneous factors, we use two strategies to address this issue:
\begin{enumerate}
    \item Thanks to efficient algorithms~\cite{solutionsSODA, vanderHoeven2002, vanderHoeven2010} for computing power series solutions of the original system~\eqref{eq:state_space}, we have an efficient randomized membership test for ideal $\mathcal{I}$.
    This membership test allows us to remove extraneous factors after each resultant computation and evade accumulation of these factors during repeated resultant computation.
    \item Our algorithm may perform a change of variables in the original system.
    Therefore, some of these extraneous factors are extracted \emph{before} computing the resultant (see Section~\ref{sec:var_change}), hence speeding up 
     the computation substantially (Table~\ref{table:change}).
\end{enumerate}


\subsection{Identifiability} 
In a parametric ODE model (that is, \eqref{eq:state_space} in which coefficients involve unknown scalar parameters), a parameter or a function of parameters is called \emph{structurally globally identifiable} (in what follows, just ``identifiable'') if its value can be uniquely determined from the input-output data, assuming the absence of noise and sufficiently exciting inputs.
This identifiability property is a natural prerequisite for practical parameter estimation, and, 
therefore, it is an important step in the experimental design process.
The problem of assessing identifiability has been studied since the 1970-s~\cite{bellman-astrom-70}.
Since then, a number of different approaches have been proposed (see \cite{comparison, HOPY2020} and references therein) and several software packages and webapps have been developed~\cite{DAISY, COMBOS, LFCBBCH,webapp}.
However, many models of practical interest remain out of reach with the existing tools.

One popular approach to assess identifiability is the \emph{input-output relations} approach proposed in~\cite{OllivierPhD} (recently used, e.g., in~\cite{Remien2021, Yeung2020, Tuncer2021}) and relies on theory of differential algebra.
It has been implemented in DAISY~\cite{DAISY}, COMBOS~\cite{COMBOS}, and Structural Identifiability Toolbox~\cite{webapp}.
First, the software computes a set of generators of the field of definition $F$ of the ideal of the input-output relations, that is, the minimal subfield of the field of rational functions in parameters sufficient to write down the generators of this ideal~\cite[Definition~2.5]{ioaaecc}.
These generators are typically taken to be the coefficients of a finite set of input-output relations, which may be computed using a characteristic set (as in DAISY~\cite{DAISY}), a Gr\"obner basis (as in COMBOS~\cite{COMBOS}), or elimination by hand (e.g., \cite{Remien2021, Tuncer2021}).
Next, the identifiability of a parameter is assessed by testing whether it belongs to $F$.
There are two important subtleties with existing algorithms following this approach:
\begin{enumerate}
    \item \textbf{Single vs. multiple experiments.} 
    In general, $F$ is equal to the field of functions identifiable from \emph{several} experiments~\cite[Theorem~19]{allident}.
    There is a sufficient condition ensuring that all the elements of $F$ can be identified from a single experiment which can be verified by computing the Wronskian of the monomials of some input-output relations~\cite[Lemma~1]{ioaaecc}.
    The condition is not checked by DAISY or COMBOS.
    However, it has been checked manually in several case studies~\cite{DVJBNP01, XiaMoog, Verdiere2005} with about a dozen of monomials and can be checked automatically by~\cite{webapp} if the number of monomials does not exceed a hundred.
    
    \item \textbf{Probability of correctness.}
    The membership test of a parameter of interest in $F$ is typically framed as an injectivity test for a so-called coefficient map.
    While this can be done deterministically using Gr\"obner bases with rational function coefficients, such a computation would be very costly.
    Because of this, both DAISY and COMBOS take a random point in the image of the map and check whether the preimage of the point is of cardinality one. 
    Such an approach may yield an incorrect result if the chosen point was not generic, and the bounds on the probability of error are not provided by DAISY or COMBOS.
\end{enumerate}

We address these issues and design a new algorithm for assessing structural parameter identifiability. 
More precisely, we:
\begin{enumerate}
    \item show how to compute the field of definition $F$ from a projection-based representation of the ideal of input-output relations (Algorithm~\ref{alg:field_def}).
    \item design and implement a practical algorithm for checking the necessary condition for the elements of $F$ to be identifiable from a single experiment, which terminates even when the input-output relations have a couple of thousands of monomials (Algorithm~\ref{alg:wronsk}).
    \item give a probability bound for testing field membership via randomization (Theorem~\ref{thm:sampling}).
\end{enumerate}
The last two items above can be used with any other elimination algorithms (e.g., used in~\cite{DAISY, COMBOS, webapp}).

The resulting algorithm is implemented in Julia in the {\sc StructuralIdentifiability} package\footnote{\url{https://github.com/SciML/StructuralIdentifiability.jl}} which is a part of the SciML (scientific machine learning) ecosystem.
Our implementation addresses the two issues outlined above and can solve problems that could not be computed before (see Table~\ref{tab:comparison_identifiability}).

\subsection{Structure of the paper}

The rest of the paper is organized as follows.
Section~\ref{sec:prelim} contains preliminaries on differential algebra, structural identifiability, and a precise description of the projection-based representation. 
Section~\ref{sec:main} describes our main theoretical results.
In Section~\ref{sec:pb_alogs} we describe and justify our algorithm for differential elimination via the projection-based representation.
In Section~\ref{sec:identifiability}, we propose an algorithm for assessing structural identifiability based on our elimination algorithm.
Section~\ref{sec:performance} describes our implementation and its performance.
The benchmark models are listed in the Appendix.


\section{Preliminaries}\label{sec:prelim}

\subsection{Differential algebra}

Throughout the paper, the ideal generated by $f_1, \ldots, f_n$ will be denoted by $\langle f_1, \ldots, f_n \rangle$.

\begin{definition}[Differential rings and fields]\label{def:diffrings}
  A {\em differential ring} $(R,\,')$ is a commutative ring with a derivation $'\!\!:R\to R$, that is, a map such that, for all $a,b\in R$, $(a+b)'=a'+b'$ and $(ab)'=a'b+ab'$. 
  A {\em differential field} is a differential ring that is a field.
  For  $i>0$,  $a^{(i)}$ denotes the $i$-th order derivative of $a \in R$.
\end{definition}

\begin{notation}
  Let $x$ be an element of a differential ring and $h \in \Znn$. We introduce
  \begin{align*}
      x^{(<h)} &:= (x, x', \ldots, x^{(h - 1)}),\\
      x^{(\infty)} &:= (x, x', x'', \ldots).
  \end{align*}
  $x^{(\leqslant h)}$ is defined analogously.
  If $\mathbf{x} = (x_1, \ldots, x_n)$ is a tuple of elements of a differential ring and $\mathbf{h} = (h_1, \ldots, h_n) \in (\Znn \cup\{\infty\})^n$, then
  \begin{align*}
      \mathbf{x}^{(< h)} &:= (x_1^{(< h)}, \ldots, x_n^{(<h)}),\\
      \mathbf{x}^{(< \mathbf{h})} &:= (x_1^{(< h_1)}, \ldots, x_n^{(< h_n)}),\\
      \mathbf{x}^{(\infty)} &:= (x_1^{(\infty)}, \ldots, x_n^{(\infty)}).
  \end{align*}
\end{notation}

\begin{definition}[Differential polynomials]
  Let $R$ be a differential ring. 
  Consider a ring of polynomials in infinitely many variables
  \[
  R[x^{(\infty)}] := R[x, x', x'', x^{(3)}, \ldots]
  \]
  and extend the derivation from $R$ to this ring by $(x^{(j)})' := x^{(j + 1)}$.
  The resulting differential ring is called \emph{the ring of differential polynomials in $x$ over $R$}.
  
  The ring of differential polynomials in several variables is defined by iterating this construction.
\end{definition}

\begin{definition}[Differential ideals]
   Let $S : =R[x_1^{(\infty)}, \ldots, x_n^{(\infty)}]$  be a ring of differential polynomials over a differential ring $R$.
   An ideal $I \subset S$ is called \emph{a differential ideal} if $a' \in I$ for every $a \in I$.
  
   One can verify that, for every $f_1, \ldots, f_s \in S$, the ideal
   \[
      \langle f_1^{(\infty)}, \ldots, f_s^{(\infty)} \rangle
    \]
    is a differential ideal.
    Moreover, this is the minimal differential ideal containing $f_1, \ldots, f_s$, and we will denote it by $\langle f_1, \ldots, f_s \rangle^{(\infty)}$.
\end{definition}

\begin{notation}[Saturation]
  Let $R$ be a ring, $I \subset R$ be an ideal, and $a \in R$.
  We introduce
  \[
      I \colon a^{\infty} := \{b \in R \mid \exists N\colon a^Nb \in I\},
  \]
  which is also an ideal in $R$.
\end{notation}


\subsection{Structural identifiability}\label{sec:identifiability_into}

 Consider an ODE system of the form
 \begin{equation}\label{eq:ODEmodel}
   \Sigma = \begin{cases}
    \mathbf{x}' = \mathbf{f}(\mathbf{x}, \bm{\mu}, \mathbf{u}),\\
    \mathbf{y} = \mathbf{g}(\mathbf{x}, \bm{\mu}, \mathbf{u}),
    \end{cases}
   \end{equation}
where 
\begin{itemize}
    \item $\mathbf{x} = (x_1, \ldots, x_n)$, $\mathbf{y} = (y_1, \ldots, y_m)$, and $\mathbf{u} = (u_1, \ldots, u_s)$ are the vectors of state, output, and input variables, respectively;
    \item $\bm{\mu} = (\mu_1, \ldots, \mu_\ell)$ is a vector of constant parameters;
    \item $\mathbf{f} = (f_1, \ldots, f_n)$ and $\mathbf{g} = (g_1, \ldots, g_m)$ are vectors of elements of $\C(\mathbf{x}, \bm{\mu}, \mathbf{u})$
\end{itemize}

By reducing $\mathbf{f}$ and $\mathbf{g}$ to the common denominator, we write $\mathbf{f} = \mathbf{F}/Q$ and $\mathbf{g} = \mathbf{G}/Q$, for $F_1, \ldots, F_n, G_1, \ldots, G_m, Q \in \C[\mathbf{x}, \bm{\mu}, \mathbf{u}]$.
Consider the (prime, see \cite[Lemma~3.2]{HOPY2020}) differential ideal
\begin{equation}\label{eq:Isigma_gens}
   I_\Sigma := \langle Qx_i' - F_i, Qy_j - G_j,\, 1\leqslant i\leqslant n,\,1\leqslant j \leqslant m\rangle^{(\infty)} \colon Q^\infty \subset \C(\bm{\mu})[\mathbf{x}, \mathbf{y}, \mathbf{u}].
\end{equation}
Note that every element of $I_\Sigma$ vanishes on every analytic solution of~\eqref{eq:ODEmodel}.

We will use the following algebraic definition of identifiability. 
Its equivalence with the general analytic definition~\cite[Definition~2.5]{HOPY2020} (see also~\cite[Definition~2.3]{OPT19}) has been established~\cite[Proposition~3.4]{HOPY2020}.

\begin{definition}[Single-experiment identifiability] 
  A rational function $h \in\C(\bm{\mu})$ will be called {\em globally (single-experiment, or SE-) identifiable } for~\eqref{eq:ODEmodel} if, 
  there exist $P_0, P_1 \in \C[\mathbf{y}^{(\infty)}, \mathbf{u}^{(\infty)}]$ such that $P_1 \not\in I_\Sigma$ and $P_1h - P_{0} \in I_\Sigma$.

  We also say that $h$ is \emph{locally SE-identifiable} if there exist a positive integer $s$ and $P_0, \ldots, P_s \in \C[\mathbf{y}^{(\infty)}, \mathbf{u}^{(\infty)}]$ such that $P_s \not\in I_\Sigma$ and $P_s h^s + P_{s - 1}h^{s - 1} + \ldots + P_0 \in I_\Sigma$.
\end{definition}

\begin{remark}
  Informally, the definition can be stated as follows:
  $h$ is globally SE-identifiable if and only if it can be expressed via inputs, outputs, and their derivatives via a formula $h = \frac{P_0(\mathbf{y}, \mathbf{u})}{P_1(\mathbf{y}, \mathbf{u})}$.
  While the existence of a formula may look like a special case of identifiability, \cite[Proposition~3.4]{HOPY2020} shows that the possibility of a unique identification of $h$ (formulated in analytic terms) is equivalent to the existence of such a formula.
\end{remark}

\begin{example}[SE-identifiability]
  Consider a harmonic oscillator:
  \[
    \Sigma\colon x_1' = -\mu x_2, \; x_2' = \mu x_1,\; y = x_1.
  \]
  Observe that
  \[
    (y - x_1)'' + (x_1' + \mu x_2)' - \mu (x_2' - \mu x_1) + \mu^2 (y - x_1) = y'' + \mu^2 y \in I_\Sigma.
  \]
  From $y'' + \mu^2 y \in I_\Sigma$ we can conclude that $h = \mu^2$ is globally identifiable (with $P_1 = y, P_0 = y''$) and $\mu$ is locally identifiable (with $s = 2$, $P_2 = y, P_1 = 0, P_0 = y''$).
\end{example}

\begin{definition}[Multi-experiment identifiability]\label{def:ident:multi}
  \begin{itemize}
      \item[]
      \item For a model $\Sigma$ and a positive integer $r$, we define \emph{the $r$-fold replica} of $\Sigma$ as
      \[
      \Sigma_r := \begin{cases}
          \mathbf{x}_i' = \mathbf{f}(\mathbf{x}_i, \bm{\mu}, \mathbf{u}_i), \;\; i = 1, \ldots, r,\\
          \mathbf{y}_i = \mathbf{g}(\mathbf{x}_i, \bm{\mu}, \mathbf{u}_i), \;\; i = 1, \ldots, r,
      \end{cases}
      \]
      where $\mathbf{x}_1, \ldots, \mathbf{x}_r, \mathbf{y}_1, \ldots, \mathbf{y}_r, \mathbf{u}_1, \ldots, \mathbf{u}_r$ are new tuples of indeterminates (note that the vector of parameters is not being replicated).
      
      \item For a model $\Sigma$, a rational function $h \in \mathbb{C}(\bm{\mu})$ is \emph{globally multi-experimental identifiable} (ME-identifiable) if there exists a positive integer $r$ such that $h(\bm{\mu})$ is globally SE-identifiable in $\Sigma_r$.
      \emph{Local ME-identifiability} is defined analogously.
  \end{itemize}
\end{definition}

\begin{definition}[Field of definition]
  Let $J \subset K[\mathbf{x}^{(\infty)}]$ (with $\mathbf{x} = (x_1, \ldots, x_n)$) be a differential ideal over a differential field $K$.
  Then the smallest differential subfield $L \subset K$ such that $J$ is generated by $J \cap L[\mathbf{x}^{(\infty)}]$ is called the field of definition of $J$.
\end{definition}

\begin{theorem}[{{\cite[Theorem~21]{allident}}}]
   The field of multi-experimental identifiable functions is generated over $\C$ by the field of definition of 
   \[
     I_{\Sigma} \cap \C(\bm{\mu}) [\mathbf{y}^{(\infty)}, \mathbf{u}^{(\infty)}].
   \]
\end{theorem}

Therefore, the problem of computing the field of ME-identifiable functions reduces to the problem of computing the field of definition of a projection of an irreducible differential-algebraic variety.
Furthermore, using a Wronskian-based criterion~\cite[Lemma~1]{ioaaecc}, one can in many cases (see Table~\ref{tab:our_performance}) establish that SE-identifiable and ME-identifiable functions coincide.

\begin{example}[ME-identifiability]
In order to illustrate the difference between the SE- and ME-identifiability, we consider the following artificial example (a version of~\cite[Example~2.14]{HOPY2020}):
\[
  \Sigma\colon x' = 0,\; y_1 = x,\; y_2 = \mu_1x + \mu_2.
\]
Since $y_1' = y_2' = 0$ and there is no algebraic relation between $y_1$ and $y_2$ modulo $I_\Sigma$, we have $I_\Sigma \cap \C[y_1^{(\infty)}, y_2^{(\infty)}] = \langle y_1, y_2 \rangle^{(\infty)}$.
Therefore, no function of $\mu_1, \mu_2$ is SE-identifiable.

On the other hand, modulo $I_{\Sigma_2}$, we will have
\[
    y_{1, 2} = \mu_1 y_{1, 1} + \mu_2 \quad\text{ and }\quad y_{2, 2} = \mu_1 y_{1, 2} + \mu_2.
\]
These equations yield a linear system for $\mu_1$ and $\mu_2$, in particular, by subtracting these equations, we obtain:
\[
  \mu_1 (y_{1, 1} - y_{2, 1}) - (y_{1, 2} - y_{2, 2}) \in I_{\Sigma_2},
\]
so $\mu_1$ is globally identifiable in $\Sigma_2$ and, thus, globally ME-identifiable in $\Sigma$ (same for $\mu_2$).
\end{example}


\subsection{Projections of an irreducible differential-algebraic variety}

\begin{notation}
\begin{itemize}
    \item[]
    \item $\Znncomp{} := \Znn \cup \{\infty\}$.
    \item For $n > 0$ and $1 \leqslant i \leqslant n$, $\mathbf{e}_{i, n}$ denotes the $i$-th basis vector in $\Z^n$.
    If $n$ is clear from the context, we will write simply $\mathbf{e}_i$.
\end{itemize}
\end{notation}

\begin{definition}[Parametric profile]
  For a prime differential ideal $P \subset K[\mathbf{x}^{(\infty)}]$ with $\mathbf{x} = (x_1, \ldots, x_n)$, a tuple $\mathbf{h} \in \Znncomp{n}$ is called \emph{a parametric profile} for $P$ if the images of $\mathbf{x}^{(< \mathbf{h})}$ form a transcendence basis of $K[\mathbf{x}^{(\infty)}] / P$.
  This notion may be viewed as a generalization of observability indices from control theory~\cite[\S 4.6]{Conte2007}.
\end{definition}

\begin{lemma}\label{lem:id_principal}
  Let $P \subset K[\mathbf{x}^{(\infty)}]$ (with $\mathbf{x} = (x_1, \ldots, x_n)$) be a prime differential ideal and $\mathbf{h} \in \Znncomp{n}$ be a parametric profile for $P$.
  Then the ideal
  \[
    P \cap K[\mathbf{x}^{(< \mathbf{h} + \mathbf{e}_i)}]
  \]
  is principal for every $1 \leqslant i \leqslant n$ such that $h_i \neq \infty$.
\end{lemma}

\begin{proof}
    Fix $i$ such that $h_i \neq \infty$.
    Let $f \in P \cap K[\mathbf{x}^{(< \mathbf{h} + \mathbf{e}_i)}]$ be an irreducible polynomial of minimal degree in $x_i^{(h_i)}$.
    Since the leading coefficient of $f$ does not belong to $P$, for every other polynomial $g \in P \cap K[\mathbf{x}^{(< \mathbf{h} + \mathbf{e}_i)}]$, the result of pseudo-division of $g$ by $f$ with respect to $x_i^{(h_i)}$ belongs to $P$.
    The minimality of $\deg_{x_i^{(h_i)}} f$ implies that this result is zero, so $g$ is pseudo-divisible by $f$ with respect to $x_i^{(h_i)}$.
    Then the Gauss lemma implies that $g$ is divisible by $f$.
\end{proof}

\begin{definition}
  For a prime differential ideal $P \subset K[\mathbf{x}^{(\infty)}]$ (where $\mathbf{x} = (x_1, \ldots, x_n)$) and its parametric profile $\mathbf{h}$, a tuple of differential polynomials $(f_1, \ldots, f_n)$ is called \emph{the projections corresponding to $\mathbf{h}$} if, for every $1 \leqslant i \leqslant n$,
  \begin{enumerate}
      \item if $h_i = \infty$, $f_i = 0$;
      \item if $h_i < \infty$, $f_i$ is the generator of $P \cap K[\mathbf{x}^{(< \mathbf{h} + \mathbf{e}_i)}]$.
  \end{enumerate}
\end{definition}

The following lemma describes how one can use the projections for differential elimination.

\begin{lemma}[Projections under projection]\label{lem:proj_proj}
  Let $(f_1, \ldots, f_n)$ be the projections of a prime differential ideal $P \subset K[\mathbf{x}^{(\infty)}]$ with respect to its profile $\mathbf{h}$.
  Assume that, for some $m \leqslant n$, differential polynomials $f_1, \ldots, f_m$ do not involve $x_{m + 1}, \ldots, x_n$ and their derivatives.
  
  Then $(h_1, \ldots, h_m)$ is a parametric profile of $P \cap K[x_1^{(\infty)}, \ldots, x_m^{(\infty)}]$, and the corresponding projections are $(f_1, \ldots, f_m)$.
\end{lemma}

\begin{proof}
  Since $x_1^{(<h_1)}, \ldots, x_m^{(<h_m)}$ are algebraically independent modulo $P$, they are algebraically independent modulo  $P \cap K[x_1^{(\infty)}, \ldots, x_m^{(\infty)}]$ as well.
  Moreover, they form a transcendence basis of $K[x_1^{(<h_1)}, \ldots, x_m^{(<h_m)}] / P \cap K[x_1^{(\infty)}, \ldots, x_m^{(\infty)}]$ because, for every $1 \leqslant i \leqslant m$, $f_i$ involves only $x_1^{(<h_1)}, \ldots, x_m^{(<h_m)}, x_i^{(h_i)}$ and thus provides an algebraic dependence between their images.
  Therefore, $(h_1, \ldots, h_m)$ is indeed a parametric profile for $P \cap K[x_1^{(\infty)}, \ldots, x_m^{(\infty)}]$. 
  Then one can see that $(f_1, \ldots, f_m)$ are the corresponding projections.
\end{proof}

\begin{remark}[Projections for ODE models]\label{rem:proj_ode}
  For an ideal $I_\Sigma \subset \C(\bm{\mu})[\mathbf{x}^{(\infty)}, \mathbf{y}^{(\infty)}, \mathbf{u}^{(\infty)}]$, corresponding to a system $\Sigma$ of the form~\eqref{eq:ODEmodel}, we have the following parametric profile (see~\cite[Lemma~3.1]{HOPY2020}):
  \begin{equation}\label{eq:profile_orig}
    (\underbrace{1, 1, \ldots, 1}_{n \text{ times}}, \underbrace{0, 0, \ldots, 0}_{m \text{ times}}, \underbrace{\infty, \ldots, \infty}_{s \text{ times}}).
  \end{equation}
  Furthermore, the corresponding projections are
  \[
  (Qx_1' - F_1, \ldots, Qx_n' - F_n, Qy_1 - G_1, \ldots, Qy_m - G_m, \underbrace{0, \ldots, 0}_{s \text{ times}}).
  \]
  By Lemma~\ref{lem:proj_proj}, one can compute the projections of the ideal of input-output relations, that is $I_\Sigma \cap \C (\bm{\mu})[\mathbf{y}^{(\infty)}, \mathbf{u}^{(\infty)}]$, by ``moving'' as many of the first $n$ ones in~\eqref{eq:profile_orig} as possible to the $y$-variables (see Figure~\ref{tab:yound_oscillator} and Example~\ref{ex:proj}).
\end{remark}

\begin{example}\label{ex:proj}
  Consider system $\Sigma$ in variables $x_1, x_2, y$ which is a harmonic oscillator with one of the coordinates being observed:
  \begin{equation}\label{eq:harm}
    x_1' = -x_2, \quad x_2' = x_1, \quad y = x_1.
  \end{equation}
  As explained in Remark~\ref{rem:proj_ode}, \eqref{eq:harm} are the projections corresponding to profile $(1, 1, 0)$.
  Since a general solution of~\eqref{eq:harm} is of the form $y = c_1 e^{it} + c_2 e^{-it}$ for arbitrary $c_1, c_2 \in \C$ and every element of $I_\Sigma$ must vanish on every solution of~\eqref{eq:harm}, there is no nonzero element in $I_\Sigma$ depending only on $y$ and $y'$, so they are algebraically independent.
  Since, modulo $I_\Sigma$, any polynomial can be reduced to a polynomial in $x_1$ and $x_2$, the transcendence degree modulo $I_\Sigma$ is 2, so $y$ and $y'$ form a transcendence basis modulo $I_\Sigma$.
  This yields a different profile $(0, 0, 2)$ for the same ideal.
  Now we compute the corresponding projections:
  \begin{itemize}
      \item the projection relating $x_1, y, y'$ is just one of~\eqref{eq:harm}: $x_1 - y = 0$;
      \item the projection relating $x_2, y, y'$ can be obtained from the previous one by differentiation: $x_2 - y' = 0$;
      \item the projection relating $y, y, y''$ is the classical equation for the harmonic oscillator: $y'' + y = 0$.
  \end{itemize}
  By Lemma~\ref{lem:proj_proj}, $y'' + y = 0$ is the projection of the ideal of input-output relations.
\end{example}

\begin{remark}[Extra components in the projection-based representation]\label{rem:extra}
    Let $(f_1, \ldots, f_n)$ be the projections of a prime differential ideal $P \subset K[\mathbf{x}^{(\infty)}]$ corresponding to a profile $\mathbf{h}$.
    Similar to~\eqref{eq:Isigma_gens}, one can define
    \begin{equation}\label{eq:Ih}
      I_{\mathbf{h}} := \langle f_1, \ldots, f_n\rangle^{(\infty)} \colon (H \cdot S)^\infty,
    \end{equation}
    where $H$ and $S$ are, respectively, the products, over all $i$'s with $f_i \neq 0$, of the leading coefficients and the derivatives of $f_i$ with respect to $x_i^{(h_i)}$.
    Since $P$ is a prime containing $I_{\mathbf{h}}$ and has the same transcendence basis as all associated primes of $I_{\mathbf{h}}$ by~\cite[Theorem 4.4]{Hubert2003}, $P$ is a prime component of $I_{\mathbf{h}}$.
    However, in general, $I_{\mathbf{h}}$ may have additional prime components as the following example shows.
    Although the example may look involved, it is an ODE version of the geometric fact that an intersection of two circular cylinders of the same radius and with intersecting axes consists of two components (ellipses).
    
    Consider an ODE system $\Sigma$ of the form~\eqref{eq:ODEmodel}:
    \[
        x_1' = \frac{1}{2}(1 + x_1^2),\quad 
        x_2' = \frac{1 - x_1^2}{1 + x_1^2},\quad
        y_1 = \frac{2x_1}{\mu(1 + x_1^2)},\quad
        y_2 = x_2.
    \]
    The corresponding ideal $I_\Sigma$ is prime, and one can show (using Algorithm~\ref{alg:socoban}) that it has a parametric profile $\mathbf{h} := (0, 0, 1, 1)$, and the corresponding projections are:
    \[
      \mu y_1 (\underline{x_1}^2 + 1) - 2 \underline{x_1} = 0, \quad \underline{x_2} - y_1 = 0, \quad \mu^2 (\underline{y_1'})^2 + \mu^2 y_1^2 - 1 = 0, \quad  (\underline{y_2'})^2 + \mu^2 y_1^2 - 1 = 0,
    \]
     where the variables of the form $x_i^{(h_i)}$ are underlined.
    By computing the difference of the last two equations, we have
    \begin{equation}\label{eq:factors}
    (\mu^2 (y_1')^2 + \mu^2 y_1^2 - 1) - ((y_2')^2 + \mu^2 y_1^2 - 1) = \mu^2 (y_1')^2 - (y_2')^2 = (\mu y_1' - y_2') (\mu y_1' + y_2') \in I_{\mathbf{h}}.
    \end{equation}
    One can check that neither of $\mu y_1' - y_2'$ and $\mu y_1' + y_2'$ belongs to $I_{\mathbf{h}}$, so it is not prime.
    Because of this phenomenon, the projections will be typically used in combination with a weak membership test for the ideal (see Section~\ref{sec:membership} for details) to distinguish between the factors as in~\eqref{eq:factors}.
    
    Furthermore, the coefficients of the projections $\mu^2 (y_1')^2 + \mu^2 y_1^2 - 1$ and $(y_2')^2 + \mu^2 y_1^2 - 1$ generate $\C(\mu^2)$ while the field of definition of the ideal of input-output relations (and, thus, by~\cite[Theorem~21]{allident}, the field of ME-identifiable functions) contains $\mu$ since $\mu y_1' - y_2' \in I_\Sigma$.
    We overcome this difficulty by designing Algorithm~\ref{alg:field_def} for computing the field of definition for a given set of projections together with a weak membership test.
\end{remark}

\begin{remark}[Projections vs. characteristic sets]\label{rem:charsets}
    Perhaps the most popular way of representing differential ideals is via characteristic sets~\cite[Section 3.3]{Hubert2003b}.
    One can check that, by~\cite[Definition~3.17]{Hubert2003b}, the set of nonzero projections $\{f_i \mid i = 1, \ldots, n, f_i \neq 0\}$ for a profile $\mathbf{h}$ of a prime differential ideal $P$ form a characteristic set of $I_{\mathbf{h}}$~\eqref{eq:Ih} with respect to the following ranking:
    \begin{equation}\label{eq:ranking}
      x_i^{(a)} > x_j^{(b)} \iff (a - h_i > b - h_j) \text{ or } (a - h_i = b - h_j \text{ and } i > j).
    \end{equation}
    However, as we have shown in Remark~\ref{rem:extra}, $I_{\mathbf{h}}$ may be not equal to $P$.
    Nevertheless, if the algebraic ideal
    \[
        \langle f_1, \ldots, f_n \rangle \colon H^\infty
    \]
    has a unique component of the top dimension, then it is prime by~\cite[Theorem~4.4]{Hubert2003}, so $I_{\mathbf{h}}$ is prime by~\cite[Theorem~4.13]{Hubert2003b} implying that $I_{\mathbf{h}} = P$.
    In this case $\{f_i \mid i = 1, \ldots, n, f_i \neq 0\}$ is characteristic set for $P$ with respect to the ranking~\eqref{eq:ranking}.
    We give and implement an algorithm for checking this (Algorithm~\ref{alg:high_dim_comp}) and it turns out that this property holds for the ideals of input-output relations for the majority of systems we have considered, see Table~\ref{tab:our_performance} (column ``Char. set?'').
    The prevalence of input-output projections which are characteristic sets makes the comparison of our algorithm with the methods based on characteristic sets in Table~\ref{tab:comparison_elimination} even more meaningful.
\end{remark}

\begin{remark}[Complexity]
  The geometric definition via the projections allows one to expect to obtain useful geometry-based bounds for the projection-based representation and the algorithms presented in this paper.
  This interesting problem is beyond the scope of the present paper.
\end{remark}


\section{Main theoretical results}\label{sec:main}

In this section, we collect the main theoretical results of the paper.
For implementation and performance of the corresponding algorithms, see Section~\ref{sec:performance}.

\subsection{Computing projections of ODE systems}

\begin{theorem}[Model $\implies$ Projections]\label{thm:project_ode}
   Consider an ODE system $\Sigma$ as in~\eqref{eq:ODEmodel} and the corresponding differential ideal $I_{\Sigma}$ (see~\eqref{eq:Isigma_gens}).
   Then Algorithm~\ref{alg:project_ode} computes a parametric profile and the corresponding projections for the ideal of input-output relations, that is, $I_{\Sigma} \cap \C(\bm{\mu})[\mathbf{y}^{(\infty)}, \mathbf{u}^{(\infty)}]$.
\end{theorem}

As we have described in Section~\ref{sec:identifiability_into}, the field of ME-identifiable functions is equal to the field of definition of $I_{\Sigma} \cap \C(\bm{\mu})[\mathbf{y}^{(\infty)}, \mathbf{u}^{(\infty)}]$.
The next theorem shows that one can compute generators of this field from the input-output projections.

\begin{theorem}[Projections $\implies$ Field of definition]\label{thm:field_def_ode}
   Consider an ODE system $\Sigma$ as in~\eqref{eq:ODEmodel} and the corresponding differential ideal $I_{\Sigma}$ (see~\eqref{eq:Isigma_gens}).
   Let $\mathbf{h}$ be a parametric profile for $I_{\Sigma} \cap \C(\bm{\mu})[\mathbf{y}^{(\infty)}, \mathbf{u}^{(\infty)}]$ with projections $(f_1, \ldots, f_n)$ such that each nonzero $f_i$ has at least one coefficient that is one.
   
   Then Algorithm~\ref{alg:field_def} computes a differential polynomial $f \in I_\Sigma$ with at least one coefficient (as a polynomial in $\{\mathbf{y}^{(\infty)}, \mathbf{u}^{(\infty)}\}$) being equal to one, so that the field of ME-identifiable functions for $\Sigma$ is generated by the coefficients of $f, f_1, \ldots, f_n$.
\end{theorem}


\subsection{Assessing identifiability}

Theorems~\ref{thm:project_ode} and~\ref{thm:field_def_ode} provide a way to compute a set of generators of the field of multi-experiment identifiable functions for a model~$\Sigma$.
However, these generators may be complicated and hard to interpret, which is what typically happens in practice.
For any given function of parameters, the following theorem allows to find with high probability of correctness whether the function belongs to the field of multi-experiment identifiable functions.

\begin{theorem}[Field of definition $\implies$ ME-identifiability]\label{thm:sampling}
   Let $f, f_1, \ldots, f_N, g \in \C[\mathbf{X}]$, where $\mathbf{X} = (X_1, \ldots, X_n)$.
   Let $0 < p < 1$ be a real number.
   We define 
   \[
     d := \max(\deg g + 1, \deg f, \deg f_1, \ldots, \deg f_N) \quad \text{ and }\quad M := \frac{6 d^{n + 3}}{1 - p}
   \]
   Let $\mathbf{a} = (a_1, \ldots, a_n)$ be integers sampled uniformly and independently from $[0, M]$.
   Consider the field $F := \C\left( \frac{f_1(\mathbf{X})}{g(\mathbf{X})}, \ldots, \frac{f_N(\mathbf{X})}{g(\mathbf{X})}  \right)$ and ideal
   \[
     I := \langle f_1(\mathbf{X}) g(\mathbf{a}) - f_1(\mathbf{a})g(\mathbf{X}), \ldots, f_N(\mathbf{X}) g(\mathbf{a}) - f_N(\mathbf{a})g(\mathbf{X}) \rangle \colon g(\mathbf{X})^{\infty} \subset \C[\mathbf{X}]
   \]
   Then
   \begin{enumerate}
       \item if $f(\mathbf{X})g(\mathbf{a}) - f(\mathbf{a})g(\mathbf{X}) \in I$, then $\frac{f(\mathbf{X})}{g(\mathbf{X})} \in F$ with probability at least $p$;
       \item if $f(\mathbf{X})g(\mathbf{a}) - f(\mathbf{a})g(\mathbf{X}) \not\in I$, then $\frac{f(\mathbf{X})}{g(\mathbf{X})} \not\in F$ with probability at least $p$.
   \end{enumerate}
\end{theorem}

\begin{remark}
\begin{itemize}
    \item[]
    \item We will apply Theorem~\ref{thm:sampling} to normalized coefficients of input-output equations (see Algorithm~\ref{alg:main}).
    Since there will be usually only a few resulting input-output equations, there will be typically few distinct denominators.
    Thus, the common denominator $g$ will not have excessively high degree.
    \item In the identifiability context, $N$ is usually a large number, so it is important that it does not appear in the formula for $M$, the size of the sampling range.
\end{itemize}
\end{remark}

\begin{remark}
  Note that it may be beneficial, instead of running the test from Theorem~\ref{thm:sampling} once, run it $\log_2 \frac{1}{1 - \varepsilon}$ times for probability $p = \frac{1}{2}$ to achieve overall probability of success $\varepsilon$.
\end{remark}

Finally, by combining Theorems~\ref{thm:project_ode}, \ref{thm:field_def_ode}, and~\ref{thm:sampling} with Algorithm~\ref{alg:wronsk} for checking the Wronskian condition of the equivalence of SE- and ME-identifiability, we obtain:

\begin{theorem}[Projection-based identifiability testing]\label{thm:ident}
   For any given ODE system $\Sigma$ as in~\eqref{eq:ODEmodel}, function of parameters $h(\bm{\mu}) \in \C(\bm{\mu})$, and a real number $0 < p < 1$, Algorithm~\ref{alg:main} returns the following information:
   \begin{itemize}
       \item whether $h(\bm{\mu})$ is globally, locally, or not ME-identifiable (correct with probability at least $p$);
       \item whether the algorithm was able to conclude that SE- and ME-identifiability coincide for $\Sigma$.
   \end{itemize}
\end{theorem}


\section{Projection-based (PB) elimination}
\label{sec:pb_alogs}

In this section, we will describe the algorithms from Theorem~\ref{thm:field_def_ode} and~\ref{thm:project_ode} and prove the theorems. 
The results and algorithms will be formulated in a more general context of prime differential ideals with membership oracles.

\subsection{Membership oracle for an ODE model}\label{sec:membership}

\begin{definition}[Weak membership oracle]\label{def:weak_membership}
    Let $P \subset K[\mathbf{x}^{(\infty)}]$ be a prime differential ideal over a differential field $K$ in variables $\mathbf{x} = (x_1, \ldots, x_n)$.
    Let $\mathcal{O}$ be an algorithm taking as input a list of differential polynomials from $K[\mathbf{x}^{(\infty)}]$ and returning, if it terminates, an element of $K[\mathbf{x}^{(\infty)}]$.
    
    Then $\mathcal{O}$ is called \emph{a weak membership oracle} for $P$ if, for every $f_1, \ldots, f_s \in K[\mathbf{x}^{(\infty)}]$ such that exactly one of $f_1, \ldots, f_s$ belongs to $P$ (denote it by $f$), we have $\mathcal{O}(f_1, \ldots, f_s) = f$.
\end{definition}

Let $\Sigma$ be an ODE model as in~\eqref{eq:ODEmodel}.
In this section, we describe an efficient sampling-based weak membership oracle for the corresponding ideal $I_{\Sigma}$ (see~\eqref{eq:Isigma_gens}).
Note that the representation~\eqref{eq:Isigma_gens} for $I_{\Sigma}$ is already a characteristic set of this ideal~\cite[Lemma~3.2]{HOPY2020}, so one can test membership in $I_{\Sigma}$ by reducing with respect to this characteristic set.
However, in practice, this may be very inefficient.

\begin{algorithm}[H]
\caption{Weak membership oracle for $I_\Sigma$}\label{alg:weak_membership}
\begin{description}[itemsep=0pt]
\item[Input ] A system~$\Sigma$ as in~\eqref{eq:ODEmodel} and polynomials $f_1, \ldots, f_s \in \C(\bm{\mu})[\mathbf{y}^{(\infty)}, \mathbf{u}^{(\infty)}]$ such that exactly one of $f_1, \ldots, f_s$ belongs to $I_\Sigma$.
\item[Output ] Index $i$ such that $f_i \in I_\Sigma$.
\end{description}

\begin{enumerate}[label = \textbf{(Step~\arabic*)}, leftmargin=*, align=left, labelsep=2pt, itemsep=0pt]
    \item Set $D := 1$, let $h$ be the maximum of the orders of $f_1, \ldots, f_s$;
    \item Repeat
    \begin{enumerate}[leftmargin=0mm, align=left, labelsep=2pt, itemsep=0pt, topsep=0pt]
        \item\label{st:choose_rand} Choose random values for parameters $\bm{\mu}$, initial conditions $\mathbf{x}$, and $\mathbf{u}^{(\leqslant h)}$ uniformly from $[1, D]$
        \item If  $Q$, the common denominator of $f_1, \ldots, f_s$, vanishes at the chosen values, go to the next iteration of the loop;
        \item\label{st:ps_sol} Compute the truncated power series solution of~$\Sigma$ up to order $h$ using the algorithm from~\cite{solutionsSODA} for the values chosen in~\ref{st:choose_rand};
        \item For each $1 \leqslant i \leqslant s$, plug the computed solution to $f_i$ and evaluate at $t = 0$;
        \item If exactly one of the $f_i$'s vanished after evaluation, return its index;
        \item Set $D := 2D$.
    \end{enumerate}
\end{enumerate}
\end{algorithm}

\begin{lemma}\label{lem:weak_mem_term}
  Algorithm~\ref{alg:weak_membership} is correct and terminates with probability one.  
\end{lemma}

\begin{proof}
  Assume that Algorithm~\ref{alg:weak_membership} terminates and returns index $i_0$.
  This implies, that the algorithm has computed a truncated power series solution of~$\Sigma$ on which all the $f_j$'s with $j \neq i_0$ do not vanish.
  Therefore, $f_j \not\in I_\Sigma$ for $j \neq i_0$, so $f_{i_0} \in I_\Sigma$.
  
  To analyze the probability of termination, let $i_0$ be the correct output.
  Let $F = \prod\limits_{j \neq i_0} f_j$.
  As has been shown in~\cite[Proof of Lemma~3.2]{HOPY2020} there exists a positive integer $N$ and nonzero $\widetilde{F} \in \C[\bm{\mu}, \mathbf{x}, \mathbf{u}^{(\leqslant h)}]$ such that $Q^N F - \widetilde{F} \in I_\Sigma$.
  We denote the total degree of $\widetilde{F}Q$ by $d$.

  Let $\bm{\mu}, \mathbf{x}^\ast, (\mathbf{u}^{(\leqslant h)})^\ast$ be the values chosen at~\ref{st:choose_rand} and $(\mathbf{x}(t), \mathbf{y}(t), \mathbf{u}(t))$ be the truncated power series solution computed at~\ref{st:ps_sol}.
  Then, if $\widetilde{F}(\bm{\mu}^\ast, \mathbf{x}(t), \mathbf{y}(t), \mathbf{u}(t))$ does not vanish at $t = 0$, then the same is true for $F(\bm{\mu}^\ast, \mathbf{x}(t), \mathbf{y}(t), \mathbf{u}(t))$.
  Since $\widetilde{F} \in \C[\bm{\mu}, \mathbf{x}, \mathbf{u}^{(\leqslant h)}]$, the value of $\widetilde{F}(\bm{\mu}^\ast, \mathbf{x}(t), \mathbf{y}(t), \mathbf{u}(t))$ at $t = 0$ is equal to $\widetilde{F}(\bm{\mu}^\ast, \mathbf{x}^\ast, (\mathbf{u}^{(\leqslant h)})^\ast)$.
  The probability of this value or $Q$ being zero does not exceed $\frac{d}{D}$ due to the Demillo-Lipton-Schwartz-Zippel lemma~\cite[Proposition~98]{Zippel}.
  Therefore, for each iteration of the main loop, the algorithm will not terminate on this iteration with the probability at most $\frac{d}{D}$.
  Therefore, the probability of termination is at least 
  \[
    1 - \frac{d}{1}\cdot \frac{d}{2}\cdot \frac{d}{4}\cdot \frac{d}{8}\cdot \ldots = 1.
  \]
\end{proof}


\subsection{``Socoban'' algorithm: changing the profile}

Lemma~\ref{lem:proj_proj} suggests that computing a parametric profile and the corresponding projections for the ideal of input-output relations of system $\Sigma$ as in~\eqref{eq:ODEmodel} can be carried out by changing the original profile~\eqref{eq:profile_orig} to a profile of the form
\[
(\underbrace{\leqslant 1, \leqslant 1, \ldots, \leqslant 1}_{n \text{ times}}, \underbrace{h_1, h_2, \ldots, h_m}_{m \text{ times}}, \underbrace{\infty, \ldots, \infty}_{s \text{ times}}),
\]
where none of the projections corresponding to the outputs contains a state variable.
This is done by Algorithm~\ref{alg:project_ode} below which in turn uses Algorithm~\ref{alg:socoban} for performing an elementary ``carrying'' step.

\begin{algorithm}[H]
\caption{Projecting an ODE system}\label{alg:project_ode}
\begin{description}[itemsep=0pt]
\item[Input ] A system~$\Sigma$ as in~\eqref{eq:ODEmodel}.
\item[Output ] A parametric profile and the corresponding projections for $I_{\Sigma} \cap \C(\bm{\mu})[\mathbf{y}^{(\infty)}, \mathbf{u}^{(\infty)}]$.
\end{description}

\begin{enumerate}[label = \textbf{(Step~\arabic*)}, leftmargin=*, align=left, labelsep=2pt, itemsep=0pt]
    \item Let $\mathcal{O}$ be the weak membership oracle for $I_{\Sigma}$ given by Algorithm~\ref{alg:weak_membership};
    \item Consider the following profile $\mathbf{h}$ and the corresponding projections $\mathbf{f}$ for $I_\Sigma$ (see Remark~\ref{rem:proj_ode})
    \begin{align*}
        \mathbf{h} &:= (\underbrace{1, 1, \ldots, 1}_{n \text{ times}}, \underbrace{0, 0, \ldots, 0}_{m \text{ times}}, \underbrace{\infty, \ldots, \infty}_{s \text{ times}}),\\
        \mathbf{p} &:= (Qx_1' - F_1, \ldots, Qx_n' - F_n, Qy_1 - G_1, \ldots, Qy_m - G_m, 0, \ldots, 0).
    \end{align*}
    \item\label{step:while} While there exist $n + 1 \leqslant i \leqslant n + m$ and $1 \leqslant j \leqslant n$ such that $x_j$ appears in $p_i$:
    \begin{enumerate}[leftmargin=0mm, align=left, labelsep=2pt, itemsep=0pt, topsep=0pt]
        \item Apply Algorithm~\ref{alg:socoban} with $P = I_\Sigma$, $\mathbf{h} = \mathbf{h}$, $\mathbf{f} = \mathbf{p}$, $\mathcal{O} = \mathcal{O}$, $i = i$, and $j = j$.
        \item Set $\mathbf{h} := \mathbf{h} - \mathbf{e_j} + \mathbf{e_i}$ and $\mathbf{p}$ to be the projections returned by Algorithm~\ref{alg:socoban}.
    \end{enumerate}
    \item Return $(h_{n + 1}, \ldots, h_{n + m + s})$ and $(p_{n + 1}, \ldots, p_{n + n + s})$.
\end{enumerate}
\end{algorithm}

\begin{remark}
  It often happens that there are several possible pairs $(i, j)$ at~\ref{step:while} of the algorithm.
  The choice of the one to process may have significant impact on the performance of the algorithm. 
  The choice made in our implementation is to have $\deg_{x_j} p_i$ as small as possible, similar quantities (e.g., the degree of the prolongation of $p_i$) are used to break the ties.
  Systematic study of choice heuristics is an interesting question for future research. 
\end{remark}

\begin{proposition}[Proof of Theorem~\ref{thm:project_ode}]
  Algorithm~\ref{alg:project_ode} is correct and terminates with probability one.
\end{proposition}

\begin{proof}
  Correctness follows from Lemma~\ref{lem:proj_proj}.
  Now we will prove termination.
  After each iteration of the while loop at~\ref{step:while}, the sum $h_{n + 1} + \ldots + h_{n + m}$ increases by one.
  Since it is bounded by $n$, there will be at most $n$ iterations.
  Moreover, each iteration will terminate with probability one due to Proposition~\ref{prop:socoban}.
\end{proof}

\begin{algorithm}[H]
\caption{``Socoban'' algorithm}\label{alg:socoban}
\begin{description}[itemsep=0pt]
\item[Input ] 
\begin{itemize}
    \item[]
    \item A prime differential ideal $P \subset K[\mathbf{x}^{(\infty)}]$ defined by
      \begin{itemize}
        \item a parametric profile $\mathbf{h}$ and corresponding projections $\mathbf{f} = (f_1, \ldots, f_n)$;
        \item and a weak membership oracle $\mathcal{O}$.
      \end{itemize}
    \item $1 \leqslant i \neq j \leqslant n$ such that $f_i$ involves $x_j$ and $h_j = 1$.
\end{itemize}

\item[Output ] the projections of $P$ corresponding to $\mathbf{h} - \mathbf{e_j} + \mathbf{e_i}$.
\end{description}

\begin{enumerate}[label = \textbf{(Step~\arabic*)}, leftmargin=*, align=left, labelsep=2pt, itemsep=0pt]
    \item Set $f := f_i'$.
    \item For $\ell \in \{1, \ldots, n\} \setminus \{i\}$ such that $h_\ell < \infty$ do
    \begin{enumerate}\label{step:elim_xl}
        \item Compute $g_1, \ldots, g_r$, the squarefree factorization over $K$ of $\Res_{x_\ell^{(h_\ell)}} (f, f_\ell)$;
        \item \label{step:elim_der_choose} Set $f := \mathcal{O}(g_1, \ldots, g_r)$.
    \end{enumerate}
    \item Set $\tilde{f}_j := f_i$ and $f_i := f$.
    \item For $\ell \in \{1, \ldots, n\} \setminus \{j\}$ such that $h_\ell < \infty$ do
    \begin{enumerate}\label{step:elim_xj}
        \item \label{step:elim_xj_res} Compute $g_1, \ldots, g_r$, the squarefree factorization over $K$ of $\Res_{x_j} (\tilde{f}_j, f_\ell)$;
        \item \label{step:elim_xj_choose} Set $\tilde{f}_\ell := \mathcal{O}(g_1, \ldots, g_r)$.
    \end{enumerate}
    \item Return $(\tilde{f}_1, \ldots, \tilde{f}_n)$.
\end{enumerate}
\end{algorithm}

\begin{lemma}\label{lem:repres}
Let $\mathbb{A}$ be an arbitrary integral domain, 
        \[
            r \in \mathbb{A}[z] \setminus \mathbb{A},
            R  \in \mathbb{A}[w_1, w_2, \ldots, w_t, z, v], 
        \]
        \[
            s_1 \in \mathbb{A}[w_1] \setminus \mathbb{A}, 
            s_2 \in \mathbb{A}[w_2] \setminus \mathbb{A},
            \ldots, 
            s_t \in \mathbb{A}[w_t] \setminus \mathbb{A} 
        \]
be non-zero polynomials such that $r$ is irreducible.
$r \nmid \operatorname{lc}_v(R)$ (where $\operatorname{lc}_v(R)$ denotes the leading coefficient of $R$ with respect to the variable $v$),
and that $\operatorname{lc}_v(R)$ does not involve $w_1, \ldots, w_t$.

\medskip

Then $\Res_{z}(r, \Res_{w_t}(s_{t}, \ldots \Res_{w_2}(s_2, \Res_{w_1}(R, s_1)) \ldots )) \neq 0$.
\end{lemma}

\begin{proof}
By replacing $\mathbb{A}$ with its algebraic closure, we will further assume that $\mathbb{A}$ is algebraically closed. 
Then each of $r, s_1, \ldots, s_t$ factors into linear factors over $\mathbb{A}$.
We will denote the roots of $r$ and $s_i$ for $1 \leqslant i \leqslant t$ in $\mathbb{A}$ by $\alpha_1, \ldots, \alpha_{\deg r}$ and $\beta_{i, 1}, \ldots, \beta_{i, \deg s_i}$, respectively.
Applying iteratively the formula for the resultant in terms of roots~\cite[Chapter~3, \S 1, Ex. 10]{CLO}, we show that the resultant of interest is equal to
\[
  C \prod R(\beta_{1, i_i}, \beta_{2, i_2}, \ldots, \beta_{t, i_j}, \alpha_j, v), 
\]
where $C$ is a power product of the leading coefficients of $r, s_1, \ldots, s_t$. 
Since none of $w_1, \ldots, w_t$ appears in $\operatorname{lc}_v(R)$ and $\operatorname{lc}_v(R)$ does not vanish at any of $\alpha_1, \ldots, \alpha_{\deg r}$, each $R(\beta_{1, i_i}, \beta_{2, i_2}, \ldots, \beta_{t, i_j}, \alpha_j, v)$ is a nonzero polynomial in $v$.
Thus, the resultant is nonzero.
\end{proof}

\begin{proposition}\label{prop:socoban}
  Algorithm~\ref{alg:socoban} terminates with probability one and is correct.
\end{proposition}

\begin{proof}
  By Lemma~\ref{lem:weak_mem_term}, each of \ref{step:elim_der_choose} and \ref{step:elim_xj_choose} terminates with probability 1. 
  Therefore algorithm~\ref{alg:socoban} terminates with probability one.
  
  Each of the computed polynomials $\tilde{f}_1, \ldots, \tilde{f}_n$ is obtained as an irreducible factor of a polynomial, so is irreducible.
  Moreover, $\tilde{f}_1, \ldots, \tilde{f}_n$ are obtained from elements of $P$ by a chain of resultant computations, so they also belong to $P$.
  Therefore, in order to prove the correctness, it remains to prove that the resultants computed on steps~\ref{step:elim_xl} and~\ref{step:elim_xj} are non-zero.
   We apply Lemma~\ref{lem:repres} with
  \begin{align*}
      &(w_1, \ldots, w_t) = ({x_\ell}^{(h_\ell)}: \ell \in \{1, \ldots, n\} \setminus \{i\},\; h_\ell < \infty),\qquad z = x_j,\qquad v = {x_i}^{(h_i + 1)},\\
      &(s_1, \ldots, s_t) = (f_\ell : \ell \in \{1, \ldots, n\} \setminus \{i\},\; h_\ell < \infty),\quad r = f_i,\quad R = f_i',\quad \mathbb{A} = K[\mathbf{x}^{(<\mathbf{h} - \mathbf{e_j} + \mathbf{e_i})}],
  \end{align*}
where all values are taken at the beginning of the algorithm.
  
  Let us verify that the conditions of Lemma~\ref{lem:repres} are indeed satisfied. Since $\mathbf{h}$ is a parametric profile, $\deg_{{x_\ell}^{(h_\ell)}} f_\ell > 0$, so $s_k \in \mathbb{A}[w_k] \setminus \mathbb{A}, k = 1, \ldots, t$. 
  Since $f_i$ involves $x_j$, $r \in \mathbb{A}[z] \setminus \mathbb{A}$.
  $r = f_i$ is irreducible by the primality of $P$.
  $\deg_{{x_i}^{(h_i)}} f_i > \deg_{{x_i}^{(h_i)}} \operatorname{lc}_{{x_i}^{(h_i + 1)}}(f_i')$, so $r \nmid \operatorname{lc}_v(R)$.
  $f_i$ does not involve ${x_\ell}^{(h_\ell)}: \ell \in \{1, \ldots, n\} \setminus \{i\}$, thus by properties of the Lie derivative, $\operatorname{lc}_v(R) = \operatorname{lc}_{{x_i}^{(h_i + 1)}}(f_i')$ does not involve $(w_1, \ldots, w_t) = ({x_\ell}^{(h_\ell)}: \ell \in \{1, \ldots, n\} \setminus \{i\})$.
  
  Then the result of the computation performed by~\ref{step:elim_xl} and the iteration of~\ref{step:elim_xj} with $\ell = i$ will divide the resultant from Lemma~\ref{lem:repres} and, thus, will be nonzero.
  
  For the remaining iterations (i.e. $\ell \neq i$) of the loop in~\ref{step:elim_xj}, $\Res_{x_j}(f_i, f_\ell)$ will be nonzero since both $f_i$ and $f_\ell$ are irreducible and $f_i \neq f_\ell$ (they have different degree in $x_i^{(h_i)}$).
\end{proof}


\subsection{Checking the uniqueness of the top-dimensional component}

Algorithm~\ref{alg:high_dim_comp} below will be one of the key ingredients for Section~\ref{sec:comp_field_def}.

\begin{proposition}\label{prop:top_dim_comp}
  \begin{enumerate}
      \item[]
      \item If Algorithm~\ref{alg:high_dim_comp} returns \textbf{True}, then there is a unique prime component $\mathfrak{p} \subset I$ with $\mathfrak{p}\cap \mathbb{Q}[\mathbf{y}] = 0$.
      \item There exist positive constants $C_0, C_1$ depending only on $q_1, \ldots, q_m$ and the degree of $q$ such that, if there is a unique prime component $\mathfrak{p} \subset I$ with $\mathfrak{p}\cap \mathbb{Q}[\mathbf{y}] = 0$ and $N > C_0$, then Algorithm~\ref{alg:high_dim_comp} returns \textbf{True} with probability at least $1 -\frac{C_1}{\sqrt[3]{N}}$.
  \end{enumerate}
    
\end{proposition}

In the proof of the proposition, we will use the following quantitative version of the Hilbert irreducibility theorem due to Cohen~\cite{Cohen} (see also~\cite{CD17}).

\begin{lemma}[{{Follows from~\cite[Theorem~2.1]{Cohen}}}]\label{lem:effective_hilbert_irreducibility}
  For every polynomial $p \in \Q[x_1, \ldots, x_n, t]$, there exist positive constants $C_0, C_1$ such that: for every positive integer $N > C_0$:
  \[
  P(\text{for every $\Q$-irreducible factor $q$ of $p$: } q(a_1, \ldots, a_n, t) \text{ is irreducible in }\Q[t]) > 1 - C_1 / \sqrt[3]{N},
  \]
  where integers $a_1, \ldots, a_n$ are sampled independently uniformly at random from $[-N, N]$
\end{lemma}

\begin{algorithm}[H]
\caption{Checking the uniqueness of the top-dimensional component}\label{alg:high_dim_comp}
\begin{description}[itemsep=0pt]
\item[Input ] a positive integer $N$ and squarefree polynomials $q_1, \ldots, q_m, q$ such that
\begin{enumerate}
    \item $q \in \Q[\mathbf{x}, \mathbf{y}]$, where $\mathbf{x} = (x_1, \ldots, x_n)$ and $\mathbf{y} = (y_1, \ldots, y_m)$;
    \item $q_i \in \Q[\mathbf{x}, y_i]$ and $\deg_{y_i} q_i > 0$ for every $1 \leqslant i \leqslant m$;
    \item $I \cap \mathbb{Q}[\mathbf{x}] = 0$, where $I := \langle q_1, \ldots, q_m, q \rangle$.
\end{enumerate}
\item[Output ] \textbf{True} or \textbf{False}, for precise interpretation (parametrized by $N$), see Proposition~\ref{prop:top_dim_comp}
\end{description}

\begin{enumerate}[label = \textbf{(Step~\arabic*)}, leftmargin=*, align=left, labelsep=2pt, itemsep=0pt]
    \item[\emph{Part 1: Specializing}]

    \item For each $1 \leqslant i \leqslant m$, compute $r_i \in \Q[\mathbf{x}]$, the discriminant of $q_i$ w.r.t. $y_i$.
    Define $R = \prod_{i=1}^m r_i$.
    
    \item\label{step:sample_ai} Sample integers $\mathbf{a} = (a_1, \ldots, a_n)$ independently uniformly at random from $[-N, N]$.
    
    \item If $R(\mathbf{a}) = 0$, return \textbf{False}.
    
    \item Set $\tilde{q}_i(y_i) := q_i(a_1, \ldots, a_n, y_i)$ for every $1 \leqslant i \leqslant m$ and $\tilde{q}(\mathbf{y}) := q(a_1, \ldots, a_n, \mathbf{y})$.
    
    \item[\emph{Part 2: Checking the primality of the specialization}]
    
    \item\label{step:sample_bi} Sample integers $\mathbf{b} = (b_1, \ldots, b_m)$ independently uniformly at random from $[-N, N]$.
    
    \item\label{step:compute_p} Use Gr\"obner bases to compute the minimal polynomial $p(z)$ of the linear form $b_1y_1 + \ldots + b_m y_m$ modulo zero-dimensional ideal $\widetilde{I} := \langle \tilde{q}_1, \ldots, \tilde{q}_m, \tilde{q} \rangle \subset \Q[\mathbf{y}]$.
    
    \item If $p$ is irreducible over $\mathbb{Q}$ and $\deg p = \dim_{\mathbb{Q}} \mathbb{Q}[\mathbf{y}] / \widetilde{I}$, return \textbf{True}. Otherwise, \textbf{False}.
\end{enumerate}
\end{algorithm}

\begin{proof}[Proof of Proposition~\ref{prop:top_dim_comp}]
  Consider zero-dimensional ideals $J := \langle q_1, \ldots, q_m, q \rangle \subset \mathbb{Q}(\mathbf{x})[\mathbf{y}]$ and $J_0 := \langle q_1, \ldots, q_m\rangle \subset \mathbb{Q}(\mathbf{x})[\mathbf{y}]$.
  The prime components of $J$ are in bijection with the prime components $\mathfrak{p} \subset I$ with $\mathfrak{p} \cap k[\mathbf{x}] = 0$, so there is a unique such component iff $J$ is prime.
  
  Consider a linear form $\ell = \lambda_1 y_1 + \ldots + \lambda_m y_m$ with $\bm{\lambda} := (\lambda_1, \ldots, \lambda_m)$ being new indeterminates.
  Then $\ell$ separates the roots of $J_0$ in $\overline{k(\mathbf{x})}$.
  Let $P, P_0 \in \mathbb{Q}(\bm{\lambda}, \mathbf{x})[t]$ be the monic minimal polynomials for $\ell$ modulo $J$ and $J_0$, respectively.
  We have $P \mid P_0$.
  Since $q_1, \ldots, q_m$ form a Gr\"obner basis, $P_0$ can be represented as a characteristic polynomial of a matrix with the entries in $\mathbb{Q}[\bm{\lambda}, \mathbf{x}, 1 / R(\mathbf{x})]$, so its coefficients belong to this ring.
  Then the same holds for $P$ by the Gauss lemma.
  \cite[Theorem~3]{Decker} implies that $J$ is prime iff $P$ is irreducible.
  
  Let $\mathbf{a} = (a_1, \ldots, a_n)$ be the integers sampled at~\ref{step:sample_ai}.
  $R(\mathbf{a}) \neq 0$ implies that the total multiplicity of solutions of $\widetilde{I}_0 := \langle \tilde{q}_1, \ldots, \tilde{q}_m\rangle$ in $\overline{\mathbb{Q}}$ is the same as the number of solutions of $J_0$.
  Thus, the minimal polynomial of $\ell$ modulo $\widetilde{I}_0$ is equal to $P_0|_{\mathbf{x} = \mathbf{a}}$.
  Then the minimal polynomial of $\ell$ modulo $\widetilde{I}$ is divisible by $P|_{\mathbf{x} = \mathbf{a}}$.
  
  Assume that the algorithm returned \textbf{True}.
  By~\cite[Theorem~3]{Decker}, we have that $\widetilde{I}$ is prime.
  Then $P|_{\mathbf{x} = \mathbf{a}}$ is irreducible, so
  $P$ is irreducible as well, and thus $J$ is prime.
  This proves the first part of the proposition.
   
   Now we will prove the second part.
   Assume that $J$ is prime.
   Let $\mathbf{b} = (b_1, \ldots, b_m)$ be the integers sampled at~\ref{step:sample_bi}.
   There is a polynomial, say $T_1(\bm{\lambda})$, such that $T_1(\mathbf{b}) \neq 0$ implies that $z = b_1 y_1 + \ldots + b_m y_m$ separates the roots of $J_0$.
   We will now assume that $T_1(\mathbf{b}) \neq 0$.
   The multiplication map by $q$ in the quotient ring $\mathbb{Q}(\mathbf{x})[\mathbf{y}] / J_0$ can be written as a matrix with the entries in $\mathbb{Q}[\mathbf{x}, 1 / R(\mathbf{x})]$.
   The dimension of its kernel is equal to the number of roots of $J$.
   Let $T_2(\mathbf{x})$ be the numerator of any maximal minor of this matrix.
   If $T_2(\mathbf{a}) \neq 0$, the number of roots of $\widetilde{I}$ is the same as of $J$.
   Then the minimal polynomial of $\ell$ modulo $\widetilde{I}$ is or the same degree as $P_{\mathbf{x} = \mathbf{a}}$, so it is equal to $P_{\mathbf{x} = \mathbf{a}}$.
   We will now assume that $T_2(\mathbf{a}) \neq 0$.    
   Then the polynomial $p(z)$ computed at the~\ref{step:compute_p} is equal to $P\vert_{\mathbf{x} = \mathbf{a}, \bm{\lambda} = \mathbf{b}}(z)$.
   Therefore, if $R(\mathbf{a}) T_1(\mathbf{b}) T_2(\mathbf{a}) \neq 0$ and the substitution $\mathbf{x} = \mathbf{a}, \bm{\lambda} = \mathbf{b}$ preserves the irreducibility of $P$, the algorithm will return \textbf{True}.
   Note that $R$ and $T_1$, and $P_0$ depend only on $q_1, \ldots, q_m$, and not on $q$ while the degree of $T_2$ can be bounded by a function of $q_1, \ldots, q_m$ and $\deg q$.
   We apply Lemma~\ref{lem:effective_hilbert_irreducibility} to $P_0(z)$ (with $t = z$), let $c_0$ and $c_1$ be the corresponding constants.
   Then, by Demillo-Lipton-Schwartz-Zippel lemma~\cite[Proposition~98]{Zippel}, the probability of $R(\mathbf{a})T_1(\mathbf{b})T_2(\mathbf{a}) \neq 0$ and $P\vert_{\mathbf{x} = \mathbf{a}, \bm{\lambda} = \mathbf{b}}(z)$ being irreducible is at least
   \[
     1 - \frac{\deg T_1 + \deg T_2 + \deg R}{N} - \frac{c_1}{\sqrt[3]{N}}, \quad \text{for }N > c_0.
   \]
   Setting $C_0 = c_0$ and $C_1 = c_1 + \deg T_1 + \deg T_2 + \deg R$ finishes the proof.
\end{proof}



\subsection{Computing the field of definition}\label{sec:comp_field_def}

Before presenting and justifying Algorithm~\ref{alg:field_def} for computing generators of a field of definition, we prove several useful properties of the field of definition.

\begin{lemma}\label{lem:field_def_basic}
  Let $I$ be an ideal of the polynomial ring $K[\mathbf{x}]$ in variables $\mathbf{x} = (x_1, \ldots, x_n)$ over a field $K$ of characteristic zero. Let $F$ be the field of definition of $I$.
  Then
  \begin{enumerate}
      \item If $I$ is a principal ideal, and its generator $f$ has at least one coefficient equal to one, then $F$ is generated by the coefficients of $f$.
      \item For any $\mathbb{Q}$-linear forms $y_1, \ldots, y_r$ in $x_1, \ldots, x_n$, the field of definition of $I \cap K[y_1, \ldots, y_r]$ is a subfield of $F$.
      \item Let $I$ be radical and $1 \leqslant s \leqslant n$. Consider $J$, the union of prime components $\mathfrak{p}$ of $I$ such that $\mathfrak{p} \cap \overline{K}[x_1, \ldots, x_s] = 0$.
      Then the field of definition of $J$ is a subfield of $F$.
  \end{enumerate}
\end{lemma}

\begin{proof}
   \begin{enumerate}
       \item Since $I$ is generated by $f$, $F$ is contained in the field generated by the coefficients of $f$.
       Conversely, $I \cap F[\mathbf{x}]$ must contain at least one polynomial of degree $\deg f$, and this polynomial will be proportional to $f$.
       Since normalizing the coefficients does not increase the field generated by the coefficients, the coefficients of $f$ belong to $F$.
       
       \item Since $I \cap F[y_1, \ldots, y_r] \subset I \cap F[\mathbf{x}]$ and $I \cap F[\mathbf{x}]$ spans $I$ as a $K$-vector space, $I \cap F[y_1, \ldots, y_r]$ must span $I \cap K[y_1, \ldots, y_r]$ as a $K$-vector space.
       Therefore, $F$ contains the field of definition of $I \cap K[y_1, \ldots, y_r]$.
       
       \item It is sufficient to show that any automorphism $\alpha\colon \overline{K} \to \overline{K}$ over $F$ fixes $J$ set-wise.
       Since $\alpha$ permutes the primes of $I$ and $\mathfrak{p} \cap \overline{K}[x_1, \ldots, x_s] \iff \alpha(\mathfrak{p}) \cap \overline{K}[x_1, \ldots, x_s],$
       $\alpha$ fixes $J$ as well.
   \end{enumerate}
\end{proof}

\begin{lemma}\label{lem:diff_ideal_trunc}
   Let $P \subset K[\mathbf{x}^{(\infty)}]$ be a prime differential ideal with a parametric profile $\mathbf{h}$ and the corresponding projections $(f_1, \ldots, f_n)$.
   We define (cf.~\eqref{eq:hcirc})
   \[
      h_i^{\circ} := \begin{cases}
        h_i, \text{ if } f_i \neq 0,\\
        \max\limits_{1 \leqslant j \leqslant n}\operatorname{ord}_{x_i}f_j, \text{ if } f_i = 0.
      \end{cases}
      \quad \text{ and }\quad \mathbf{h}^\circ := (h_1^\circ, \ldots, h_n^\circ).
    \]
   Then $P \cap K[\mathbf{x}^{(\leqslant \mathbf{h})}]$ is generated by $P \cap K[\mathbf{x}^{(\leqslant \mathbf{h}^\circ)}]$.
\end{lemma}

\begin{proof}
   We consider the images of $\mathbf{x}^{(\infty)}$ modulo $P$ and denote them by the same symbols.
   The lemma will follow from the fact that $\mathbf{x}^{(\leqslant \mathbf{h})} \setminus \mathbf{x}^{(\leqslant \mathbf{h}^\circ)}$ are algebraically independent over $\mathbf{x}^{(\leqslant \mathbf{h}^\circ)}$.
   To prove this, we assume the contrary.
   Let $\mathbf{h}^\ast := \mathbf{h}^\circ - \sum\limits_{f_i \neq 0}\mathbf{e_i}$.
   Then $\mathbf{x}^{(\leqslant \mathbf{h}^\ast)}$ form a transcendence basis of $\mathbf{x}^{(\leqslant \mathbf{h}^\circ)}$ over $K$ modulo $P$.
   Therefore, there is an algebraic dependence between
   \[
     \mathbf{x}^{(\leqslant \mathbf{h}^\ast)} \cup (\mathbf{x}^{(\leqslant \mathbf{h})} \setminus \mathbf{x}^{(\leqslant \mathbf{h}^\circ)}) = \mathbf{x}^{(< \mathbf{h})},
   \]
   but this contradicts the fact that $\mathbf{h}$ is a parametric profile.
\end{proof}

\begin{algorithm}[H]
\caption{Computing the differential field of definition}\label{alg:field_def}
\begin{description}[itemsep=0pt]
\item[Input ]  A prime differential ideal $P \subset \Q(\mathbf{p})[\mathbf{x}^{(\infty)}]$, where $\mathbf{p} = (p_1, \ldots, p_s)$ are transcendental constants, defined by
      \begin{itemize}
        \item a parametric profile $\mathbf{h}$ and the corresponding projections $(f_1, \ldots, f_n)$;
        \item and a weak membership oracle $\mathcal{O}$.
      \end{itemize}
\item[Output ] a differential polynomial $f \in P$ such that the coefficients of $f, f_1, \ldots, f_n$ (after normalizing so that every nonzero polynomial has at least one coefficient $1$) generate the field of definition of $P$.
\end{description}

\begin{enumerate}[label = \textbf{(Step~\arabic*)}, leftmargin=*, align=left, labelsep=2pt, itemsep=0pt]
    \item Introduce a new variable $z$, set $N := 1$.
    \item Relabel $x_1, \ldots, x_n$ so that there exists $1 \leqslant n_0 \leqslant n$ such that $h_1, \ldots, h_{n_0} < \infty$ and $h_{n_0 + 1} = \ldots = h_n = \infty$.
    \item Clear the denominators in $f_1, \ldots, f_{n_0}$ so that $f_1, \ldots, f_{n_0} \in \Q[\mathbf{p}, \mathbf{x}^{(\infty)}]$.
    \item Let $\ell$ be the product of the leading terms of $f_i$ w.r.t. $x_i^{(h_i)}$ for every $1 \leqslant i \leqslant n_0$.
    \item For every $1 \leqslant i \leqslant n$, define 
    \begin{equation}\label{eq:hcirc}
      h_i^{\circ} := \begin{cases}
        h_i, \text{ if } i \leqslant n_0,\\
        \max\limits_{1 \leqslant j \leqslant n_0}\operatorname{ord}_{x_i}f_j, \text{ if } i > n_0.
      \end{cases}
      \quad \text{ and }\quad \mathbf{h}^\circ := (h_1^\circ, \ldots, h_n^\circ).
    \end{equation}
    \item Repeat
    \begin{enumerate}
        \item\label{step:sample_projection} Let $a_1, \ldots, a_{n_0}$ be integers sampled from $[1, N]$ independently, uniformly at random.
        \item Set $f := z - a_1 x_1^{(h_1)} - \ldots - a_{n_0} x_n^{(h_{n_0})}$.
        \item\label{step:iter_resultant} For each $j \in \{1, \ldots, n_0\}$ do $f := \Res_{x_\ell^{(h_j)}} (f, f_j)$.
        \item Let $\tilde{f}$ be the result of applying $z \to a_1 x_1^{(h_1)} + \ldots + a_{n_0} x_n^{(h_{n_0})}$ to $f$.
        \item Compute $g_1, \ldots, g_r$, the squarefree factorization of $\tilde{f}$ over $\mathbb{Q}$.
        \item\label{step:ftilde_final} Set $\tilde{f} = \mathcal{O}(g_1, \ldots, g_r)$.
        \item Apply Algorithm~\ref{alg:high_dim_comp} to the ideal $\langle f_1, \ldots, f_{n_0}, \tilde{f} \rangle \subset \Q[\mathbf{p}, \mathbf{x}^{(\leqslant \mathbf{h}^\circ)}]$ with $N = N$ and $\mathbf{y} = (x_1^{(h_1)}, \ldots, x_{n_0}^{(h_{n_0})})$.
        \item\label{step:return} If the algorithm returns \textbf{True}, return $\tilde{f}$. 
        Otherwise, set $N := 2N$.
    \end{enumerate}
\end{enumerate}
\end{algorithm}

\begin{lemma}\label{lem:field_def_diff}
   Let $P \subset K[\mathbf{x}^{(\infty)}]$ be a prime differential ideal with a parameteric profile $\mathbf{h}$, and let $\mathbf{h}^\circ$ be defined in the same way as in Lemma~\ref{lem:diff_ideal_trunc}.
   Then the differential field of definition of $P$ is generated (as a differential field) by the field of definition of $P \cap K[\mathbf{x}^{(\leqslant \mathbf{h}^\circ)}]$.
\end{lemma}

\begin{proof}
   Let the differential field of definition of $P$ be $F$ and the differential field generated by the field of definition of $P \cap K[\mathbf{x}^{(\leqslant \mathbf{h})}]$ be $F_0$.
   Since $P \cap F[\mathbf{x}^{(\leqslant \mathbf{h})}]$ must generate $P \cap K[\mathbf{x}^{(<\mathbf{h})}]$, we have $F_0 \subset F$.
   
   After reordering $x_1, \ldots, x_n$ if necessary, we will further assume that there exists $1 \leqslant n_0 \leqslant n$ such that $h_1, \ldots, h_{n_0} < \infty$ and $h_{n_0 + 1} = \ldots = h_n = \infty$.
   Let $f_1, \ldots, f_n$ be the projections corresponding to $\mathbf{h}$ such that each nonzero $f_i$ is normalized to have at least one coefficient that is one.
   Then the first part of Lemma~\ref{lem:field_def_basic} implies that the coefficients of $f_1, \ldots, f_{n_0}$ belong to $F_0$.
   We introduce the following ranking~\cite[Definition~3.1]{Hubert2003} on $K[\mathbf{x}^{(<\infty)}]$ (as in Remark~\ref{rem:charsets}):
   \[
     x_i^{(a)} \prec x_j^{(b)} \iff (a - h_i < b - h_j) \text{ or } (a - h_i = b - h_j \text{ and } i < j)
   \]
   for every $i \neq j$.
   Then $f_1, \ldots, f_{n_0}$ is a weak differential triangular set~\cite[Defintion~3.7]{Hubert2003b} w.r.t. this ranking.
   Let $g \in P$. 
   By performing a partial reduction~\cite[Algorithm~3.12]{Hubert2003b} of $g$ with respect to $f_1, \ldots, f_{n_0}$, we obtain $\overline{g} \in P\cap K[\mathbf{x}^{(\leqslant \mathbf{h})}]$ and $S$, a product of powers of the separants of $f_1, \ldots, f_{n_0}$ such that 
   \[
     Sg - \sum c_{i, j}f_i^{(j)} = \overline{g}
   \]
   for some $c_{i, j} \in K[\mathbf{x}^{(\infty)}]$.
   Let $\{e_\lambda\}_{\lambda \in \Lambda}$ be a $F_0$ basis of $K$, and, for every $\lambda\in \Lambda$, we denote by $g_\lambda$ and $\overline{g}_\lambda$ the corresponding coordinates of $g$ and $\overline{g}$, respectively.
   Since the coefficients of $S$ and all the derivatives of $f_1, \ldots, f_n$ belong to $F_0$, for every $\lambda\in \Lambda$, we have $Sg_\lambda - \overline{g}_\lambda \in \langle f_1, \ldots, f_n\rangle^{(\infty)}$.
   Since $\overline{g} \in K[\mathbf{x}^{(\leqslant \infty)}]$, we have $\overline{g}_{\lambda} \in P$.
   Since $s \not\in P$ and $P$ is prime, we have $g_\lambda \in P$.
   Therefore, $P$ is generated by $P \cap F_0[\mathbf{x}^{(\infty)}]$, so $F_0 = F$.
\end{proof}

\begin{lemma}\label{lem:proj_distinct}
   Consider an $n + m$-dimensional affine space $\mathbb{A}^{n + m}$ over an algebraically closed field $k$ with coordinates $x_1, \ldots, x_n, y_1, \ldots, y_m$.
   Let $X \subset \mathbb{A}^{n + m}$ be an $n$-dimensional variety such that $X$ projects dominantly onto $x_1, \ldots, x_n$.
   Then there is a nonzero polynomial $P \in K[z_1, \ldots, z_m]$ for some $K \supset k$ such that, for every $\mathbf{a} = (a_1, \ldots, a_m) \in \Q^m$ with  $P(\mathbf{a}) \neq 0$, the images of the irreducible components of $X$ with respect to the projection 
   $\pi_{\mathbf{a}}(\mathbf{x}, \mathbf{y}) := (\mathbf{x}, a_1y_1 + \ldots + a_m y_m)$
   are distinct.
\end{lemma}

\begin{proof}
   The fact that the images of the components of $X$ with respect to $\pi_{\mathbf{a}}$ are distinct is equivalent to the fact that the linear form $a_1y_1 + \ldots + a_m y_m$ separates the solutions in $\overline{k(\mathbf{x})}$ of the zero-dimensional ideal $J$ generated by $I$ in $k(\mathbf{x})[\mathbf{y}]$.
   This condition can be written as a system of algebraic inequalities, so $P$ can be taken to the be the product of these inequalities.
\end{proof}

\begin{proposition}[Proof of Theorem~\ref{thm:field_def_ode}]\label{prop:alg_field_def}
     Algorithm~\ref{alg:field_def} terminates with probability one and is correct.
\end{proposition}

\begin{proof}
   We will start with proving \emph{the correctness} of the algorithm.
   Assume that the algorithm has terminated after sampling numbers $a_1^\ast, \ldots, a_{n_0}^\ast$ at~\ref{step:sample_projection}.
   Let $z^\ast := a_1^\ast x_1^{(h_1)} + \ldots + a_{n_0}^\ast x_n^{(h_{n_0})}$.
   Since $x_i^{(h_i)}$ is algebraic over $\mathbb{Q}[\mathbf{p}, \mathbf{x}^{(< \mathbf{h})}]$ modulo $P$ for every $1 \leqslant i \leqslant n_0$, $z^{\ast}$ is algebraic modulo $P$ over this ring as well.
   Since $\mathbf{p}, \mathbf{x}^{(< \mathbf{h})}$ are $\mathbb{Q}$-algebraically independent modulo $P$, the ideal
   \[
     P \cap \mathbb{Q}[\mathbf{p}, \mathbf{x}^{(<\mathbf{h})}, z^\ast]
   \]
   is a prime principal ideal.
   Since $\tilde{f}$ computed at~\ref{step:ftilde_final} belongs to this ideal and is irreducible, $\tilde{f}$ is the generator of the ideal.
   Therefore, it generates the prime principal ideal
   \[
     P \cap \mathbb{Q}(\mathbf{p})[\mathbf{x}^{(<\mathbf{h})}, z^\ast].
   \]
   Let $F$ be the field of definition of $P$ (since $\Q(\mathbf{p})$ is a constant field, this is the same as the differential field of definition) and $F_0$ be the field generated by the coefficients of $f_1, \ldots, f_{n_0}, \tilde{f}$ after normalizing at least one of the coefficients to be one.
   Lemma~\ref{lem:field_def_diff} implies that $F$ also equals to the field of definition of $P \cap \Q(\mathbf{p})[\mathbf{x}^{(\leqslant \mathbf{h})}]$ which equals to the field of definition of $P_0 := P \cap \Q(\mathbf{p})[\mathbf{x}^{(\leqslant \mathbf{h}^\circ)}]$ due to Lemma~\ref{lem:diff_ideal_trunc}.
   Combining parts 1 and 2 of Lemma~\ref{lem:field_def_basic}, we deduce that $F_0 \subset F$.
   On the other hand, since the ideal $\langle f_1, \ldots, f_{n_0}, \widetilde{f}\rangle \subset \Q[\mathbf{p}, \mathbf{x}^{(\leqslant \mathbf{h}^\circ)}]$ has a unique component of codimension $n_0$, this component is $P_0$, so $F \subset F_0$ due to the part 3 of Lemma~\ref{lem:field_def_basic}.
   \emph{Thus, the correctness of the algorithm is proved.}
   
   Now we will prove that \emph{the algorithm terminates with probability one}.
   The polynomial $f$ computed at~\ref{step:iter_resultant} will be always nonzero due to Lemma~\ref{lem:repres}.
   Let $d$ be the maximum of the degrees of $f_1, \ldots, f_{n_0}$.
   Then the degree of $f$ will not exceed $2^{n_0}d^{n_0}$, so the same is true for $\widetilde{f}$.
   Therefore, Proposition~\ref{prop:top_dim_comp} implies that there exist constants $C_0$ and $C_1$ such that, if the ideal $\langle f_1, \ldots, f_{n_0}, \widetilde{f} \rangle$ of $\Q[\mathbf{p}, \mathbf{x}^{(\leqslant \mathbf{h}^\circ)}]$ has a unique component projecting dominantly on $\{\mathbf{p}, \mathbf{x}^{(\leqslant \mathbf{h}^\circ)}\} \setminus \{x^{(h_1)}, \ldots, x^{(h_{n_0})}\}$ and $N > C_0$, then the probability that the algorithm will return the result at~\ref{step:return} is at least $1 - \frac{C_1}{\sqrt[3]{N}}$.
   
   We apply Lemma~\ref{lem:proj_distinct} to the variety $X$ consisting of all the components of $\langle f_1, \ldots, f_{n_0}\rangle  \subset \Q[\mathbf{p}, \mathbf{x}^{(\leqslant \mathbf{h}^\circ)}]$ projecting dominantly on $\{\mathbf{p}, \mathbf{x}^{(\leqslant \mathbf{h}^\circ)}\} \setminus \{x^{(h_1)}, \ldots, x^{(h_{n_0})}\}$ with
   \[
   \mathbf{x} = \{\mathbf{p}, \mathbf{x}^{(\leqslant \mathbf{h}^\circ)}\} \setminus \{x^{(h_1)}, \ldots, x^{(h_{n_0})}\}\quad \text{ and }\quad \mathbf{y} = \{x^{(h_1)}, \ldots, x^{(h_{n_0})}\},
   \]
   and denote the polynomial provided by the lemma by $Q$.
   We denote the degree of $Q$ by $d_0$.
   Consider an iteration of the ``repeat'' loop such that $N > C_0$.
   We denote the integers sampled at~\ref{step:sample_projection} by $a_1^\ast, \ldots, a_{n_0}^\ast$.
   Assume that $Q(\mathbf{a}^\ast) \neq 0$.
   Since $\widetilde{f}$ is the defining polynomial of the projection $\pi_{\mathbf{a}^\ast}$ (in the notation of Lemma~\ref{lem:proj_distinct}) of the variety defined by $P \cap \Q[\mathbf{p}, \mathbf{x}^{(\leqslant \mathbf{h}^\circ)}]$, and this variety is one of the dominantly projected components of $X$, $Q(\mathbf{a}^\ast) \neq 0$ implies that $\widetilde{f}$ does not vanish on any other dominantly projecting component of $X$.
   Therefore, the ideal $\langle f_1, \ldots, f_{n_0}, \widetilde{f}\rangle$ has a unique such component, so the algorithm will return with the probability at least $1 - \frac{C_1}{\sqrt[3]{N}}$.
   
   The Demillo-Lipton-Schwartz-Zippel lemma~\cite[Proposition~98]{Zippel} implies that $Q(\mathbf{a}^\ast) = 0$ with probability at most $\frac{d_0}{N}$.
   Therefore, for each iteration of the repeat loop with $N > C_0$ the probability of this iteration not being the last one is at most
   \[
     1 - \left(1 - \frac{d_0}{N}\right) \left( 1 - \frac{C_1}{\sqrt[3]{N}}\right) = \frac{d_0}{N} + \frac{C_1}{\sqrt[3]{N}} - \frac{d_0C_1}{N^{4/3}}.
   \]
   This value will eventually become less that $\frac{1}{2}$, so the probability of the algorithm not terminating will not exceed $\frac{1}{2} \cdot \frac{1}{2} \cdot \ldots = 0$.
\end{proof}


\subsection{Factoring resultants \emph{before} computing}\label{sec:var_change}

The algorithms described in Sections~\ref{sec:membership}-\ref{sec:comp_field_def} give a complete procedure for computing a representation of the projection of an ODE system and the generators of the field of definition of this projection.
These algorithms rely heavily on the resultant computation (e.g., \ref{step:elim_xj_res} in Algorithm~\ref{alg:socoban}).
It is well-known that resultants may give extraneous factors (especially the repeated ones~\cite{repeated}), and we remove these factors using the weak membership test (see Definition~\ref{def:weak_membership}) for the ideal of interest.
In this section we would like to present a method allowing to remove these factors \emph{before} computing the resultant.
Although the method applies only to special cases, it turned out to be very powerful in the differential elimination context (see Table~\ref{table:change}).

\begin{example}[Simple motivating example]\label{ex:factor}
   Consider polynomials $f = ax + bc$ and $g = bx + ac$.
   The resultant $\Res_x(f, g)$ can be computed using the Sylvester matrix as follows:
   \[
     \begin{vmatrix}
        a & bc\\
        b & ac
     \end{vmatrix} = (a^2 - b^2)c = c(a - b) (a + b).
   \]
   Assume that the weak membership test says that $a + b$ is the polynomial in the ideal of interest while $a - b$ and $c$ are extraneous factors.
   Note that the factor $c$ is the gcd of the last column of the matrix, so it can be factored out before the determinant computation.
   Assume also that we know \emph{in advance} that there will be the factor $a - b$.
   If $a - b = 0$, then the common root of $f$ and $g$ will be $x = -c$. 
   Let us make a change of variables shifting this root to be zero, that is, $x \to x - c$:
   \[
       f = ax + bc \to ax + bc - ac, \quad g = bx + ac \to bx + ac - bc.
   \]
   Since such a linear shift does not change the resultant, we can use the Sylvester matrix of the new polynomials for the resultant computation. This will be
   \[
       \begin{vmatrix}
          a & (b - a)c \\
          b & (a - b)c
       \end{vmatrix} = (a - b)c  \begin{vmatrix}
          a & -1\\
          b & 1
       \end{vmatrix},
   \]
   so the extraneous factors are removed before the determinant computation.
\end{example}

For factoring out $c$ in Example~\ref{ex:factor}, we used a simple \emph{observation}: if we compute a resultant of two polynomials $f$ and $g$ using a matrix representation, then, before computing the determinant, we can factor out gcds of rows and columns.
Not all extraneous factors of the resultant can be eliminated this way (e.g., $a - b$ in Example~\ref{ex:factor}), and we propose a method to force some of them to appear as such gcd by making a variables change before constructing the matrix.
We formalize the transformation from Example~\ref{ex:factor} as a lemma:
\begin{lemma}\label{lem:var_change_general}
    Let $k$ be a field of characteristic zero.
    Let $f, g \in k[\mathbf{a}, x]$ and $A \in k[\mathbf{a}]$ be such that $A(\mathbf{a}) \mid \Res_x(f, g)$.
    Assume that there exist $B, C \in k[\mathbf{a}]$ such that $\gcd(A, C) = 1$ and the numerators of $f(\mathbf{a}, B / C)$ and $g(\mathbf{a}, B / C)$ are divisible by $A$ (that is, $B/C$ is a common root of $f$ and $g$ modulo $A$).
    Then each entry of the last column of the Sylvester matrix for
    \[
        C^{\deg_x f}f(\mathbf{a}, x - B/C) \quad \text{ and }\quad C^{\deg_x g}g(\mathbf{a}, x - B/C)
    \]
    is divisible by $A(\mathbf{a})$.
\end{lemma}

\begin{proof}
    The last column contains only two nonzero entries: the constant terms of the considered polynomials.
    The constant terms of the new polynomials are divisible by the denominators of $f(\mathbf{a}, B / C)$ and $g(\mathbf{a}, B / C)$, respectively, which are in turn divisible by $A$.
\end{proof}

The way the situation described in Lemma~\ref{lem:var_change_general} occurs in our computation is described in the following proposition
\begin{proposition}\label{prop:var_change_special}
     Consider $f, g, h \in k[\mathbf{a}, x, y]$ such that $f$ is of the form $A(\mathbf{a}) xy + B(\mathbf{a}) x + C(\mathbf{a})y + D(\mathbf{x})$, and $\gcd(F, A) = 1$, where $F := A(\mathbf{a})D(\mathbf{a}) - B(\mathbf{a})C(\mathbf{a})$.
     Then
     \begin{enumerate}
         \item $F \mid \Res_{y}(\Res_{x}(f,g), \Res_{x}(f,h))$;
         \item the numerators of $\Res_{x}(f,h)$ and $\Res_{x}(f,g)$ evaluated at $y = B/A$ are divisible by $F$.
     \end{enumerate}
\end{proposition}

\begin{proof}
    To prove the first part, we consider any point $\mathbf{a}^\ast$ over the algebraic closure of $k$ such that $F(\mathbf{a}^\ast) = 0$ and $A(\mathbf{a}^\ast) \neq 0$.
    Then we can write
    \begin{equation}\label{eq:f_factor}
      f(\mathbf{a}^\ast, x, y) = \left(x + \frac{C(\mathbf{a}^\ast)}{A(\mathbf{a}^\ast)}\right) (A(\mathbf{a}^\ast) y + B(\mathbf{a}^\ast))
    \end{equation}
    Thus both $\Res_x(f, g)$ and $\Res_x(f, h)$ at $\mathbf{a}^\ast$ will be divisible by $A(\mathbf{a}^\ast) y + B(\mathbf{a}^\ast)$, so their resultant will vanish.
    Since $F(\mathbf{a}) = 0, A(\mathbf{a}) \neq 0$ is Zariski dense in $F(\mathbf{a}) = 0$, $F$ divides the resultant.
    
    To prove the second part, consider the factorization~\eqref{eq:f_factor} in the localization of $k[\mathbf{a}] / \langle F \rangle$ with respect to $A$ (which is well-defined since $\gcd(A, F) = 1$) with $\mathbf{a}^\ast$ being the image of $\mathbf{a}$ under the canonical projection.
    Since the images of $\Res_{x}(f,h)$ and $\Res_{x}(f,g)$ vanish at $B(\mathbf{a}^\ast) / A(\mathbf{a}^\ast)$, the numerators of their values at $B(\mathbf{a}) / A(\mathbf{a})$ belong to $\langle F\rangle$, so are divisible by~$F$.
\end{proof}

In order to explain the relation of Proposition~\ref{prop:var_change_special} to Algorithms~\ref{alg:project_ode} and~\ref{alg:socoban}, we observe that $\tilde{f}_1, \ldots, \tilde{f}_{j - 1}, \tilde{f}_{j + 1}, \ldots, \tilde{f}_n$ in the output of Algorithm~\ref{alg:socoban} are divisors of a resultant of some polynomial with $\tilde{f}_j$ with respect to $x_j$.
The later runs of Algorithm~\ref{alg:socoban} at~\ref{step:while} of Algorithm~\ref{alg:project_ode} will compute resultants of these polynomials and their derivatives, this yielding repeated resultants as in Proposition~\ref{prop:var_change_special}.
Therefore, if we see that $\tilde{f}_j$ is of the form as in Proposition~\ref{prop:var_change_special} with $x = x_j$ and $y = x_s$ for some $s$, we perform the change of variable $x_s$ as in Lemma~\ref{lem:var_change_general} (and its derivative is updated correspondingly).
Although, the further resultant computations will be performed not with $\tilde{f}_j$ but with factors of their derivatives (so the multiplicity of the common root may be lower), the efficiency gain of such a variable change is substantial as shown in Table~\ref{table:change} below.

\begin{table}[!htbp]
    \centering
    \begin{tabular}{|l|c|c|}
    \hline
        Model & Time without change & Time with change \\
    \hline
        SIWR model (Example~\ref{ex:SIWR}) & $26\,\text{s.}$ & $3\,\text{s.}$ \\
    \hline
        Pharmacokinetics (Example~\ref{ex:pharm}) & $93\,\text{s.}$ & $18\,\text{s.}$ \\
    \hline
        SEAIJRC model (Example~\ref{ex:seaijrc}) & $> 5\,\text{h.}$ & $29\,\text{s.}$ \\
    \hline
    \end{tabular}
    \caption{Efficiency gain by using the change of variables described in Section~\ref{sec:var_change}}
    \label{table:change}
\end{table}

\begin{remark}
In practice, most multiplicities of $A(\mathbf{a})D(\mathbf{a}) - B(\mathbf{a})C(\mathbf{a})$ in the resultant will be factored out from the Sylvester matrix, although there is no guarantee that all of them can be eliminated this way.
\end{remark}

\begin{remark}
    Similar but more complicated variable changes can be devised for other forms of $f$. Unfortunately, in practice, while these variable changes help eliminate some extraneous factors, they introduce others that are usually even more complicated. 
    This can already be seen in the $f = A(\mathbf{a})xy+B(\mathbf{a})x+C(\mathbf{a})y+D(\mathbf{a})$ case: while helping to eliminate the extraneous factor $A(\mathbf{a})D(\mathbf{a}) - B(\mathbf{a})C(\mathbf{a})$, the variable change actually introduces a new factor $A(\mathbf{a})$. Fortunately this new factor is simple enough so that the variable change is actually beneficial for the algorithm efficiency.
\end{remark}

\section{Assessing structural identifiability using the PB-representation}\label{sec:identifiability}

\subsection{Overview}\label{sec:ident_overview}
In this section we will describe how the algorithms for projection-based (PB) representation from Section~\ref{sec:pb_alogs} can be used to assess structural identifiability of the parameters (or some functions of them) in an ODE model~\eqref{eq:ODEmodel}.

The general idea of this computation follows the lines of the approach via input-output equations~\cite{OllivierPhD}.
We will give an outline of the algorithm using the notation introduced in Section~\ref{sec:identifiability_into}.
\begin{enumerate}[label = \textbf{(Step~\arabic*)}, leftmargin=*, align=left, labelsep=2pt, itemsep=0pt]
    \item\label{step:field_def} \emph{Find multi-experiment identifiable functions.} \cite[Theorem~21]{allident} implies that the field of all multi-experiment identifiable functions of a model $\Sigma$~\eqref{eq:ODEmodel} is equal to the field of definition of the following elimination ideal
    \[
        P := I_{\Sigma} \cap \C(\bm{\mu}) [\mathbf{y}^{(\infty)}, \mathbf{u}^{(\infty)}].
    \]
    We use Algorithm~\ref{alg:project_ode} to compute the projection-based representation of $P$, and then apply Algorithm~\ref{alg:field_def} to this representation and the weak membership oracle for $P$ provided by the original ODE system (see Algorithm~\ref{alg:weak_membership}) to compute polynomials $f, f_1, \ldots, f_n$ whose coefficients (after normalizing so that every nonzero polynomial has at least one coefficient that is one) generate the field of definition of $P$.
    \item \emph{Check: ME-identifiable = SE-identifiable ?} 
    \cite[Lemma~1]{ioaaecc} (together with~\cite[Theorem~21]{allident}) provides a sufficient condition for single-experiment identifiable functions to coincide with the multi-experiment ones.
    This condition yields the equality in many cases (for example, in all but one benchmarks we use in this paper).
    However, checking this criterion requires  computing the rank of certain Wronskian, and such algorithms (e.g., \cite[Remark~22]{allident}) have limited efficiency.
    They typically cannot compute the cases in which the order of the Wronskian reaches 100 and thus cannot be used in our case since the order of the Wronskian in the example considered in this paper reaches 2600 (e.g., for Example~\ref{ex:pharm}).
    We develop a fast algorithm for performing this computation described in Section~\ref{sec:wronskian}.
    
    \item \emph{Check identifiability.}
    We use Theorem~\ref{thm:sampling} to check whether a parameter or a given function of parameters belong to the field of definition of $P$ using the generators of this field computed at~\ref{step:field_def}.
\end{enumerate}

The rest of the section is organized as follows.
Section~\ref{sec:membership_field} contains the proof of Theorem~\ref{thm:sampling} which yields a probabilitic algorithm for field membership.
Section~\ref{sec:wronskian} describes an algorithm for checking whether multi-experiment and single-experiment identifiable functions coincide.
Finally, in Section~\ref{sec:identifiability_algo}, we give a complete algorithm for assessing structural identifiability detailing the above outline.


\subsection{Assessing field membership}\label{sec:membership_field}

In this section, we will prove Theorem~\ref{thm:sampling} and give a randomized Algorithm~\ref{alg:field_membership} for testing membership in rational function fields using it.

\begin{lemma}\label{lem:good_annihilator}
    Let $X \subset \mathbb{A}^{n + 1}$ be an irreducible algebraic variety over $\C$.
    We denote the coordinate functions in the ambient space by $x_1, \ldots, x_{n + 1}$.
    Assume that $x_{n + 1}$ is algebraic over $x_1, \ldots, x_n$ modulo $\I(X)$ of degree $d$.
    Then there exists a polynomial $P \in \C[x_1, \ldots, x_n]$ such that
    \begin{itemize}
        \item $P \in \I(X)$;
        \item $L \not\in \I(X)$, where $L$ is the leading coefficient of $P$ with respect to $x_{n + 1}$;
        \item $\deg_{x_{n + 1}}P = d$ and $\deg P \leqslant 2 \deg X$.
    \end{itemize}
\end{lemma}

\begin{proof}
   Relabeling $x_1, \ldots, x_n$ if necessary, we will further assume that there exists $r \leqslant n$ such that $x_1, \ldots, x_r$ form a transcendence basis of $\C[x_1, \ldots, x_{n + 1}]$ modulo $\I(X)$.
   By~\cite[Theorem~2]{DS04} $X$ can be represented by a triangular set $P_{n + 1}, P_n, \ldots, P_{r + 1}$ with respect to the ordering $x_{n + 1} > \ldots > x_1$ such that, for every $r < i \leqslant n + 1$,
   \begin{itemize}
       \item $P_i \in \C[x_1, \ldots, x_i]$;
       \item the total degree of $P_i$ with respect to $x_1, \ldots, x_r$ does not exceed $\deg X$;
       \item for every $r < j < i$, we have $\deg_{x_j} P_i < \deg_{x_j} P_j$.
   \end{itemize}
   For every $r < i \leqslant n + 1$, we set $d_i := \deg_{x_i} P_i$.
   Due to the construction of this triangular set~\cite[Definition~2]{DS04} ($P_i$'s correspond to $N_i$'s in the notation of~\cite{DS04}), we have that 
   \begin{itemize}
       \item none of the leading coefficients of $P_{n + 1}, \ldots, P_{r + 1}$ belongs to $\I(X)$;
       \item the degrees $d_{n + 1}, \ldots, d_{r + 1}$ are the degrees of the polynomials in the lexicographic Gr\"obner basis with $x_{n + 1} > \ldots > x_{r + 1}$ of the localization of $\I(X)$ with respect to $\C[x_1, \ldots, x_r]$ (see~\cite[Assumption~1]{DS04}).
   \end{itemize}
   The latter implies that $d_{n + 1} = d$ and the degree of the variety defined by the localization is equal to $d_{n + 1}\ldots d_{r + 1}$.
   Therefore, $d_{n + 1}\ldots d_{r + 1} \leqslant \deg X$.
   Thus, we can bound the total degree of $P_{n + 1}$ as follows
   \[
     \deg P_{n + 1} \leqslant \deg X + d_{n + 1} + (d_n - 1) + \ldots + (d_{r + 1} - 1) \leqslant \deg X + \prod\limits_{i = r + 1}^{n + 1}d_i \leqslant 2\deg X.
   \]
   Hence, we can take $P = P_{n + 1}$.
\end{proof}


\begin{proposition}\label{prop:fiber_cardinality}
     Let $X \subset \mathbb{A}^{n + 1}$ be an irreducible algebraic variety over $\C$.
     Let $\pi\colon \mathbb{A}^{n + 1} \to \mathbb{A}^n$ be the projection to the first $n$ coordinates, and assume that the generic fiber of the restriction $\pi|_X$ is finite. 
     Then there exists a hypersurface $H \subset \mathbb{A}^{n + 1}$ of degree at most $4\deg X$ not containing $X$ with the following properties.
     \begin{enumerate}
         \item If the size of the generic fiber of $\pi|_X$ is one, then, for every $p \in X \setminus H$, we have 
         \[
           \I(X) + \I(\pi(p) \times \C) = \I(p).
         \]
         \item If the size of the generic fiber of $\pi|_X$ is finite but greater than $1$, then, for every $p \in X \setminus H$, we have $|\pi^{-1}(\pi(p)) \cap X| > 1$.
     \end{enumerate}
\end{proposition}

\begin{proof}
   Let $d$ be the size of the generic fiber of $\pi|_X$. 
   We will give separate constructions for $H$ for the cases $d = 1$ and $d > 1$.
   
   If $d = 1$, then we apply the last part of~\cite[Lemma~4.3]{HOPY2020} with $X = X$ and $r = 1$, and obtain a desired hypersurface $H \subset \mathbb{A}^{n + 1}$.
   
   Now assume that $d > 1$. 
   We denote the coordinate functions by $x_1, \ldots, x_{n + 1}$.
   We apply Lemma~\ref{lem:good_annihilator} to $X$ and obtain a polynomial $P$ with the leading coefficient with respect to $x_{n + 1}$ denoted by $L$.
   We set $Q := L \cdot \frac{\partial P}{\partial x_{n + 1}}$.
   Note that $\deg Q < 2\deg P = 4\deg X$ and $Q$ does not vanish on $X$.
   We will prove that the hypersurface $H$ defined by $Q = 0$ satisfied the requirements of the proposition.
   We denote the images of $x_1, \ldots, x_{n + 1}$ in the ring of regular functions of $X$ (that is, in $\C[x_1, \ldots, x_{n + 1}] / \I(X)$) by $a_1, \ldots, a_{n + 1}$.
   Let $R := \C[a_1, \ldots, a_n]$.
   Then $a_{n + 1}$ is algebraic of degree $d$ over $R$.
   Consider a polynomial $p(X) := P(a_1, \ldots, a_n, X) \in R[X]$.
   Then $\deg p = d$ and $p(a_{n + 1}) = 0$.
   We define $\ell = L(a_1, \ldots, a_n) \in R$, the leading coefficient of $p$.
   Then \cite[Propostion~3.2]{Hubert2003} implies that
   \begin{equation}\label{eq:sat_repr}
     \{q(X) \in R[X] \mid q(a_{n + 1}) = 0\} = \langle p(X) \rangle \colon \ell^\infty.
   \end{equation}
   Consider a homomorphism $\varphi\colon R \to \C$ such that $\varphi(\ell) \neq 0$ and consider any root $\alpha$ of $\varphi(p)(X)$.
   \eqref{eq:sat_repr} implies that, for every $q(X) \in R[X]$ vanishing at $a_{n + 1}$, there exists $N$ such that $\ell^N q(X)$ is divisible by $p(X)$ over $R$.
   Therefore, $\alpha$ is also a root of $\varphi(q)(X)$.
   This implies that $\varphi$ can be extended to a homomorphism $R[a_{n + 1}] \to \C$ by setting $\varphi(a_{n + 1}) = \alpha$.
   In other words, $(\varphi(x_1), \ldots, \varphi(x_n), \alpha) \in X$.
   
   Now we consider any point $\mathbf{x}^\ast = (x_1^\ast, \ldots, x_{n + 1}^\ast)$ such that $Q(\mathbf{x}^\ast) \neq 0$.
   We define a homomorphism $\varphi\colon R \to \C$ by $\varphi(x_i) = x_i^\ast$ for every $1 \leqslant i \leqslant n$.
   $Q(\mathbf{x}^\ast) \neq 0$ implies $\varphi(\ell) \neq 0$ and $\varphi(p)'(x_{n + 1}^\ast) \neq 0$.
   Therefore, $\varphi(p)(X)$ is a polynomial of degree $> 1$ having at least one simple root, so it must have at least one extra root, let us call it $x_{n + 1}^\circ$.
   The previous paragraph implies that $(x_1^\ast, \ldots, x_n^\ast, x_{n + 1}^\circ) \in X$, so $|\pi^{-1}(\pi(\mathbf{x}^\ast)) \cap X| > 1$.
\end{proof}

\begin{lemma}[Geometric resolution]\label{lem:geom_resolution}
    Let $r_1(\mathbf{X}) = \frac{f_1(\mathbf{X})}{g(\mathbf{X})}, \ldots, r_N(\mathbf{X}) = \frac{f_N(\mathbf{X})}{g(\mathbf{X})} \in \mathbb{C}(\mathbf{X})$ generate a field $F$ over $\C$.
    Set $d := \max(1 + \deg g, \deg f_1, \ldots, \deg f_N)$.
    Then there exists an invertible matrix $M \in \C^{N \times N}$ such that rational functions $\tilde{r}_i := M(r_1, \ldots, r_N)^T$ have the following properties:
    \begin{enumerate}
        \item there exists $1 \leqslant \ell \leqslant N$ such that $\tilde{r}_1, \ldots, \tilde{r}_\ell$ form a $\C$-transcendence basis of $F$;
        \item $F = \C(\tilde{r}_1, \ldots, \tilde{r}_{\ell + 1})$, where $\ell$ is defined in the previous item;
        \item there exists a nonzero polynomial $q \in \C[T_1, \ldots, T_{\ell + 1}]$ with $\deg q  \leqslant d^{\ell + 1}$ such that $q(\tilde{r}_1, \ldots, \tilde{r}_{\ell + 1}) = 0$;
        \item there exist polynomials $w_{\ell + 1}, \ldots, w_{N} \in \C[T_1, \ldots, T_{\ell + 1}]$ such that 
        \[ 
        \tilde{r}_i = \frac{w_i(\tilde{r}_1, \ldots, \tilde{r}_{\ell + 1})}{\frac{\partial q}{\partial T_{\ell + 1}}(\tilde{r}_1, \ldots, \tilde{r}_{\ell + 1})}\quad \text{ for every }\ell + 1 < i \leqslant N.
        \]
    \end{enumerate}
\end{lemma}

\begin{proof}
    Consider an ideal
    \[
      I := \{p \in \C[T_1, \ldots, T_N] \mid p(r_1, \ldots, r_N) = 0\}.
    \]
    Since $I$ a prime ideal, it admits a geometric resolution due to~\cite[Proposition~3]{GLS} which yields the matrix $M$ and polynomials $q, w_{\ell + 2}, \ldots, w_N$ such that $q$ is the minimal polynomial of $\tilde{r}_{\ell + 1}$ over $\C(\tilde{r}_1, \ldots, \tilde{r}_{\ell})$.
    All the properties of $M$ and the polynomials stated in the lemma follow from~\cite[Proposition~3]{GLS} except for the degree bound for $q$ which we will now establish.
    We write $\tilde{r}_1 = \frac{\tilde{f}_1}{g}, \ldots, \tilde{r}_{\ell + 1} = \frac{\tilde{f}_N}{g}$ for $\tilde{f}_1, \ldots, \tilde{f}_N \in \C[\mathbf{X}]$.
    We introduce new variables $\mathbf{y} = (y_1, \ldots, y_{\ell + 1})$ and consider an ideal
    \[
      J := \langle \tilde{f}_1(\mathbf{X}, \mathbf{y}) - y_1g(\mathbf{X}), \ldots, \tilde{f}_{\ell + 1}(\mathbf{X}, \mathbf{y}) - y_{\ell + 1}g(\mathbf{X}) \rangle \colon g(\mathbf{X})^\infty \subset \C[\mathbf{X}, \mathbf{y}].
    \]
    The degree of the variety defined by $J$ is $\leqslant d^{\ell + 1}$ due to~\cite[Theorem~1]{Heintz1983}.
    The ideal $J$ is the vanishing ideal of the point $(\mathbf{X}, \tilde{r}_1, \ldots, \tilde{r}_{\ell + 1})$ over $\C$.
    Therefore, its projection on the $\mathbf{y}$-coordinates will be the ideal of relations between $\tilde{r}_1, \ldots, \tilde{r}_{\ell + 1}$ over $\C$ which is principal due to independence of $\tilde{r}_1, \ldots, \tilde{r}_{\ell}$ and generated by $q$. Therefore, the degree of $q$ is bounded by~$d^{\ell + 1}$.
\end{proof}

\begin{proof}[Proof of Theorem~\ref{thm:sampling}]
    We recall the following notation from the statement of the theorem
    \[
    I := \langle f_1(\mathbf{X}) g(\mathbf{a}) - f_1(\mathbf{a})g(\mathbf{X}), \ldots, f_N(\mathbf{X}) g(\mathbf{a}) - f_N(\mathbf{a})g(\mathbf{X}) \rangle \colon g(\mathbf{X})^{\infty} \subset \C[\mathbf{X}]
    \]
    and denote $F := \C\left( \frac{f_1(\mathbf{X})}{g(\mathbf{X})}, \ldots, \frac{f_N(\mathbf{X})}{g(\mathbf{X})} \right)$.

    We start with replacing $\frac{f_1}{g}, \ldots, \frac{f_N}{g}$ with the result of applying Lemma~\ref{lem:geom_resolution} to them.
    Note that this will not change the ideal $I$ and the field $F$.
    Let integer $\ell$ and polynomials $q, w_{\ell + 2}, \ldots, w_N \in \C[T_1, \ldots, T_{\ell + 1}]$ be as in the lemma.
    
    We introduce new variables $y_1, \ldots, y_{\ell + 1}, z$, and denote by $X$ the variety in $\mathbb{A}^{n + \ell + 2}$ defined by the ideal
    \begin{equation}\label{eq:defining_X}
      \langle f_1(\mathbf{X}) - y_1g(\mathbf{X}), \ldots, f_{\ell + 1}(\mathbf{X}) - y_{\ell + 1}g(\mathbf{X}), f(\mathbf{X}) - z g(\mathbf{X}) \rangle  \colon g(\mathbf{X})^\infty \subset \C[\mathbf{X}, \mathbf{y}, z].
    \end{equation}
    Denote the projection of $\mathbb{A}^{n + \ell + 2}$ to the $(\mathbf{y}, z)$-coordinates by $\pi$.
    Let $X_0 := \overline{\pi(X)}$.
    Denote the projection $(\mathbf{y}, z) \to \mathbf{y}$ by $\pi_0$.
    Note that $\frac{f(\mathbf{X})}{g(\mathbf{X})} \in F$ if and only if the generic fiber of $\pi_0|X_0$ is of cardinality one.
    If the generic fiber of $\pi_0|_{X_0}$ is finite, let $P_0$ be the defining polynomial of the hypersurface given by Proposition~\ref{prop:fiber_cardinality} applied to variety $X_0$ and projection $\pi_0$.
    Then 
    \begin{equation}\label{eq:deg_bound_P0}
      \deg P_0 \leqslant 4 \deg X_0 \leqslant 4\deg X.
    \end{equation}
    We apply~\cite[Lemma~4.4]{HOPY2020} to the variety $X$ and projection $\pi$, and denote the resulting subvariety of $X_0$ by $Y$, we have $\deg Y \leqslant \deg X$.
    
    We will now define a polynomial $Q \in \C[y_1, \ldots, y_{\ell + 1}, z]$ depending on the geometric situation such that 
    \begin{equation}\label{eq:Qprop}
    Q\left( \frac{f_1(\mathbf{a})}{g(\mathbf{a})}, \ldots, \frac{f_{r + 1}(\mathbf{a})}{g(\mathbf{a})}, \frac{f(\mathbf{a})}{g(\mathbf{a})} \right) \neq 0 \implies \left[ \frac{f(\mathbf{X})}{g(\mathbf{X})} \in F \iff f(\mathbf{X})g(\mathbf{a}) - f(\mathbf{a}) g(\mathbf{X}) \in I \right]
    \end{equation}
    \begin{itemize}
        \item \emph{If the cardinality of the generic fiber of $\pi_0|_{X_0}$ is one.}
        We set $Q := P_0$.
        \item \emph{If the cardinality of the generic fiber of $\pi_0|_{X_0}$ if greater than one.}
        We define polynomial $P_1 \in \C[y_1, \ldots, y_\ell]$ with $\deg P_1 \leqslant \deg X$ as described below and set $Q := qP_0P_1$ (where polynomial $q$ comes from Lemma~\ref{lem:geom_resolution}, geometric resolution, as explained in the beginning of the proof).
        \begin{itemize}
            \item \emph{If the generic fiber of $\pi_0|_{X_0}$ is finite.}
            Since $Y$ is a proper subvariety of $X_0$ and the generic fiber of $\pi_0|_{X_0}$ is finite, $\overline{\pi_0(Y)}$ is a proper subvariety of $\overline{\pi_0(X_0)}$.
            Since $\overline{\pi_0(Y)}$ has degree at most $\deg Y \leqslant \deg X$, it can be defined by polynomials of degree at most $\deg X$.
            Therefore, due to~\cite[Proposition~3]{Heintz1983}, there exists a polynomial $P_1 \in \C[\mathbf{y}]$ of degree at most $\deg X$ vanishing on $\overline{\pi_0(Y)}$ but not vanishing on $\overline{\pi_0(X_0)}$.
            \item\emph{If the generic fiber of $\pi_0|_{X_0}$ is infinite.}
            Since the dimension of a fiber of $\pi_0$ does not exceed one and the dimension of a fiber is upper semicontinuous~\cite[Theorem~14.8]{eisenbud}, each fiber of $\pi_0|_{X_0}$ is one-dimensional, and, consequently, a line.
            Since $\dim Y < \dim X_0$, either $\overline{\pi_0(Y)} \subsetneq \overline{\pi_0(X)}$ or the generic fiber of $\pi_0|_{Y}$ is finite.
            \begin{itemize}
              \item If $\overline{\pi_0(Y)} \subsetneq \overline{\pi_0(X)}$, let $P_1$ be a polynomial vanishing on $\overline{\pi_0(Y)}$ but not on $\overline{\pi_0(X_0)}$.
              Since $\deg \overline{\pi_0(Y)} \leqslant \deg Y$, \cite[Proposition~3]{Heintz1983} implies that $P_1$ can be chosen so that $\deg P_1 \leqslant \deg Y \leqslant \deg X$.
              \item If $\overline{\pi_0(Y)} = \overline{\pi_0(X)}$, then the generic fiber of $\pi_0|_Y$ is finite.
              We apply~\cite[Lemma~4.3(i)]{HOPY2020} with $X = Y$ and $\pi = \pi_0$, and denote the resulting subvariety of $\overline{\pi_0(Y)} = \overline{\pi_0(X)}$ by $Z$.
              By~\cite[Lemma~4.3(i)]{HOPY2020}, $\deg Z \leqslant \deg X$, so, by~\cite[Proposition~3]{Heintz1983}, we can choose $P_1$ to vanish on $Z$ and not vanish on $\overline{\pi_0(X)}$ such that $\deg P_1 \leqslant \deg Z \leqslant \deg X$.
            \end{itemize}
        \end{itemize}
        
    \end{itemize}
    
    Consider any point $\mathbf{a} \in \C^n$ such that
    \[
        Q\left( \frac{f_1(\mathbf{a})}{g(\mathbf{a})}, \ldots, \frac{f_{r + 1}(\mathbf{a})}{g(\mathbf{a})}, \frac{f(\mathbf{a})}{g(\mathbf{a})} \right) \neq 0 \text{ and is well-defined}.
    \]
    Again, consider cases:
    \begin{itemize}
            \item \emph{the cardinality of the generic fiber of $\pi_0|_{X_0}$ is one.}
            Then we have (by the choice of $P_0$)
            \[
              z - \frac{f(\mathbf{a})}{g(\mathbf{a})} \in \I(X_0) + \left\langle y_1 - \frac{f_1(\mathbf{a})}{g(\mathbf{a})}, \ldots, y_{r + 1} - \frac{f_{r + 1}(\mathbf{a})}{g(\mathbf{a})} \right\rangle.
            \]
            Therefore,
            \[
              f(\mathbf{X}) g(\mathbf{a}) - f(\mathbf{a}) g(\mathbf{X}) \in \I(X) + \left\langle y_1 - \frac{f_1(\mathbf{a})}{g(\mathbf{a})}, \ldots, y_{r + 1} - \frac{f_{r + 1}(\mathbf{a})}{g(\mathbf{a})} \right\rangle.
            \]
            Together with~\eqref{eq:defining_X} this yields
            \[
              f(\mathbf{X}) g(\mathbf{a}) - f(\mathbf{a}) g(\mathbf{X}) \in \langle f_1(\mathbf{X}) g(\mathbf{a}) - f_1(\mathbf{a}) g(\mathbf{X}),\ldots,  f_{r + 1}(\mathbf{X}) g(\mathbf{a}) - f_{r + 1}(\mathbf{a}) g(\mathbf{X}) \rangle \colon g(\mathbf{X})^\infty \subset I.
            \]
            \item \emph{the cardinality of the generic fiber of $\pi_0|_{X_0}$ is greater than one.}
            We will show that there exists $\alpha \in \C$ such that $\alpha \neq \frac{f(\mathbf{a})}{g(\mathbf{a})}$ and 
            \[
              p := \left( \frac{f_1(\mathbf{a})}{g(\mathbf{a})}, \ldots, \frac{f_{r + 1}(\mathbf{a})}{g(\mathbf{a})}, \alpha \right) \in X_0 \setminus Y.
            \]
            Consider cases:
            \begin{itemize}
                \item \emph{if the generic fiber $\pi_0|_{X_0}$ is finite,} then the choice of $P_0$ implies that there exists $\alpha \in \C$ such that $\alpha \neq \frac{f(\mathbf{a})}{g(\mathbf{a})}$ and $p := (\frac{f_1(\mathbf{a})}{g(\mathbf{a})}, \ldots, \frac{f_{r + 1}(\mathbf{a})}{g(\mathbf{a})}, \alpha) \in X_0$.
                Since $P_1\left( \frac{f_1(\mathbf{a})}{g(\mathbf{a})}, \ldots, \frac{f_{r + 1}(\mathbf{a})}{g(\mathbf{a})} \right) \neq 0$, we have $p \not\in Y$.
                
                \item \emph{if generic fiber of $\pi_0|_{X_0}$ is infinite,} the inequation $P_1\left( \frac{f_1(\mathbf{a})}{g(\mathbf{a})}, \ldots, \frac{f_{r + 1}(\mathbf{a})}{g(\mathbf{a})} \right) \neq 0$ implies that $\pi_0^{-1}(\frac{f_1(\mathbf{a})}{g(\mathbf{a})}, \ldots, \frac{f_{r + 1}(\mathbf{a})}{g(\mathbf{a})}) \cap Y$ is finite, so one can choose a point $p$ to avoid this finite set.
            \end{itemize}
            Since $p \in X_0 \setminus Y$, there exists $\mathbf{b} \in \C^n$ such that 
            \[
              p = \left( \frac{f_1(\mathbf{b})}{g(\mathbf{b})}, \ldots, \frac{f_{r + 1}(\mathbf{b})}{g(\mathbf{b})}, \frac{f(\mathbf{b})}{g(\mathbf{b})} \right).
            \]
            Since $q\left( \frac{f_1(\mathbf{a})}{g(\mathbf{a})}, \ldots, \frac{f_{\ell + 1}(\mathbf{a})}{g(\mathbf{a})} \right) \neq 0$, the values of $\frac{f_{\ell + 2}(\mathbf{a})}{g(\mathbf{a})}, \ldots, \frac{f_{N}(\mathbf{a})}{g(\mathbf{a})}$ are uniquely defined by $\frac{f_1(\mathbf{a})}{g(\mathbf{a})}, \ldots, \frac{f_{\ell + 1}(\mathbf{a})}{g(\mathbf{a})}$.
            Therefore, since $\frac{f_i(\mathbf{a})}{g(\mathbf{a})} = \frac{f_i(\mathbf{b})}{g(\mathbf{b})}$ for $1 \leqslant i \leqslant r + 1$, we have $\frac{f_i(\mathbf{a})}{g(\mathbf{a})} = \frac{f_i(\mathbf{b})}{g(\mathbf{b})}$ for $\ell + 1 < i \leqslant N$ as well.
            Since $f(\mathbf{X})g(\mathbf{a}) - f(\mathbf{a})g(\mathbf{X})$ and $g(\mathbf{X})$ do not vanish at $\mathbf{b}$, the polynomial $f(\mathbf{X})g(\mathbf{a}) - f(\mathbf{a})g(\mathbf{X})$ does not belong to $I$.
        \end{itemize}
        
        Since the conclusion of the implication~\eqref{eq:Qprop} implies both conclusions of the theorem, it is sufficient to prove that
        \[
          P\left( Q\left( \frac{f_1(\mathbf{a})}{g(\mathbf{a})}, \ldots, \frac{f_{\ell + 1}(\mathbf{a})}{g(\mathbf{a})} \right) \neq 0\;\text{ and well-defined} \right) \geqslant p.
        \]
        We start with bounding the degree of $Q$:
        \[
          \deg Q \leqslant \deg P_0 + \deg P_1 + \deg q \leqslant 5\deg X + \deg q \leqslant 5d^{\ell + 2} + d^{\ell + 1},
        \]
        where the last inequality follows from the Bezout bound~\cite[Theorem~1]{Heintz1983}.
        Then the degree of the denominator $Q_d(\mathbf{X})$ of $Q\left(\frac{f_1(\mathbf{X})}{g(\mathbf{X})}, \ldots, \frac{f_{\ell + 1}(\mathbf{X})}{g(\mathbf{X})} \right)$ does not exceed $5d^{\ell + 3} + d^{\ell + 2}$.
        Therefore,
        \[
          P\left( Q\left( \frac{f_1(\mathbf{a})}{g(\mathbf{a})}, \ldots, \frac{f_{\ell + 1}(\mathbf{a})}{g(\mathbf{a})} \right) \neq 0\;\text{ well-defined} \right) \geqslant P( Q_d(\mathbf{a}) g(\mathbf{a}) \neq 0) \geqslant 1 - \frac{5d^{\ell + 3} + d^{\ell + 2} + d}{M}.
        \]
        where the second inequality is due to the Demillo-Lipton-Schwartz-Zippel lemma~\cite[Proposition~98]{Zippel}.
        Since $d \geqslant 2$, the latter expression is greater than $p$.
\end{proof}

\begin{algorithm}[H]
\caption{Checking membership for rational function fields}\label{alg:field_membership}
\begin{description}[itemsep=0pt]
\item[Input ] rational functions $f, f_1, \ldots, f_N \in \C(\mathbf{x})$, where $\mathbf{x} = (x_1, \ldots, x_n)$, and a real number $0 < \varepsilon < 1$;
\item[Output ] \texttt{YES} if $f \in \C(f_1, \ldots, f_N)$ and \texttt{NO} otherwise.

The result is guaranteed to be correct with probability at least $1 - \varepsilon$.
\end{description}

\begin{enumerate}[label = \textbf{(Step~\arabic*)}, leftmargin=*, align=left, labelsep=2pt, itemsep=0pt]
    \item Write $f, f_1, \ldots, f_N$, where $f = \frac{F}{Q}$ and $f_i = \frac{F_i}{Q}$ for every $i = 1, \ldots, N$ and $Q, F, F_1, \ldots, F_N \in \C[x_1, \ldots, x_n]$.
    \item Compute $d := \max(\deg Q + 1, \deg F, \deg F_1, \ldots, \deg F_N)$ and $M := 6d^{n + 3} / \varepsilon$.
    \item Sample integers $\mathbf{a} = (a_1, \ldots, a_n)$ from $[0, M]$ uniformly at random.
    \item Set $I := \langle F_1(\mathbf{x}) Q(\mathbf{a}) - F_1(\mathbf{a})Q(\mathbf{x}), \ldots, F_N(\mathbf{x}) Q(\mathbf{a}) - F_N(\mathbf{a})Q(\mathbf{x}) \rangle \colon Q(\mathbf{x})^{\infty} \subset \C[\mathbf{x}]$.
    \item Use Gr\"obner bases to check $F(\mathbf{x})Q(\mathbf{a}) - F(\mathbf{a})Q(\mathbf{x}) \in I$. If yes, return \texttt{YES}, otherwise return \texttt{NO}.
\end{enumerate}
\end{algorithm}

\begin{lemma}
    The output of Algorithm~\ref{alg:field_membership} is correct with probability at least $1 - \varepsilon$.
\end{lemma}

\begin{proof}
    Follows directly from Theorem~\ref{thm:sampling}.
\end{proof}


\subsection{Checking whether SE-identifiable = ME-identifiabile}\label{sec:wronskian}

We will use the notation from Section~\ref{sec:ident_overview}.
As described in \ref{step:field_def} of the approach outlined there, for a given model $\Sigma$~\eqref{eq:ODEmodel}, one can combine Algorithms~\ref{alg:project_ode} and~\ref{alg:field_def} in order to compute $f, f_1, \ldots, f_n$ such that the coefficients of $f, f_1, \ldots, f_n$ (after normalizing so that every nonzero polynomial has at least one coefficient is equal to one) generate the field of definition of the ideal of input-output relations, $I_\Sigma \cap \C(\bm{\mu})[\mathbf{y}^{(\infty)}, \mathbf{u}^{(\infty)}]$, which is equal to the field of all multi-experiment identifiable functions due to~\cite[Theorem~21]{allident}.
In this section, we will present an algorithm allowing one to conclude in some cases (in practice, very often) that the coefficients also generate the field of single-experiment identifiable functions.
The algorithm will be based on Lemma~\ref{lem:wronsk} below which, is a slight modification of~\cite[Lemma~1]{ioaaecc}.

\begin{definition}[Wronskian]
    Let $a_1, \ldots, a_m \in R$ be elements of a differential ring $R$. 
    Then we defined the \emph{Wronskian} of $a_1, \ldots, a_m$ to be the following matrix
    \[
    \wronsk(a_1, \ldots, a_m) := \begin{pmatrix}
    a_1 & \ldots & a_m\\
    a_1' & \ldots & a_m'\\
    \vdots & \ddots & \vdots \\
    a_1^{(m - 1)} & \ldots & a_m^{(m - 1)}
    \end{pmatrix}.
    \]
\end{definition}

\begin{lemma}\label{lem:wronsk}
    Let $\Sigma$ and $I_\Sigma$ be as above.
    Let $g \in I_\Sigma \cap \C(\bm{\mu})[\mathbf{y}^{(\infty)}, \mathbf{u}^{(\infty)}]$ be a nonzero element such that at least one of the coefficients of $g$ is $1$.
    Assume that $g$ is written as $g = z_0 + \sum\limits_{i = 1}^N a_i z_i$ such that
    $z_0, \ldots, z_N \in \C[\mathbf{y}^{(\infty)}, \mathbf{u}^{(\infty)}]$ and $a_1, \ldots, a_N \in \C(\bm{\mu})$.
    If the rank of the $\wronsk(z_0, \ldots, z_N)$ modulo $I_\Sigma$ is equal to $N$, then $a_1, \ldots, a_N$ are single-experiment identifiable.
\end{lemma}

\begin{proof}
    Since the proof of~\cite[Lemma~1]{ioaaecc} does not use the fact that $z_0, \ldots, z_N$ are monomials but only uses that they belong to $\C[\mathbf{y}^{(\infty)}, \mathbf{u}^{(\infty)}]$, the same argument proves this lemma.
\end{proof}

If one can verify that each of $f, f_1, \ldots, f_n$ satisfy the requirements of Lemma~\ref{lem:wronsk}, then one can conclude that the fields of single-experiment and multi-experiment functions coincide.
However, in the benchmarks we consider in this paper, the number $N$ in Lemma~\ref{lem:wronsk} can be as large as $2600$ (in Example~\ref{ex:pharm}).
For the problems of this size, the algorithms computing the derivatives of the outputs up to this order symbolically (such as in~\cite{allident,webapp}) are not efficient enough.
We propose Algorithm~\ref{alg:wronsk}, which, together with the enhancement from Remark~\ref{rem:massive_eval}, can compute the rank of a specialization of the Wronskian from Lemma~\ref{lem:wronsk} in several minutes for all the examples we consider.

\begin{algorithm}[H]
\caption{Checking single-experiment identifiability of coefficients}\label{alg:wronsk}
\begin{description}[itemsep=0pt]
\item[Input ] an ODE model~$\Sigma$ as in~\eqref{eq:ODEmodel} and a differential polynomial $g \in \C(\bm{\mu})[\mathbf{y}^{(\infty)}, \mathbf{u}^{(\infty)}] \cap I_{\Sigma}$ with at least one coefficient being one.
\item[Output ] \texttt{Single-experiment} if all the coefficients of $g$ are single-experiment identifiable or \texttt{Not sure} if single-experiment identifiability could not be concluded.
\end{description}

\begin{enumerate}[label = \textbf{(Step~\arabic*)}, leftmargin=*, align=left, labelsep=2pt, itemsep=0pt]
    \item Write $g = z_0 + \sum_{i = 1}^N a_i z_i$, where $z_0, \ldots, z_N \in \C[\mathbf{y}^{(\infty)}, \mathbf{u}^{(\infty)}]$ and $a_1, \ldots, a_N \in \C(\bm{\mu})$ such that $N$ is minimal possible using~\cite[Algorithm~3]{allident}.
    \item\label{step:sample_wrnsk} Let $h$ be the order of $g$, and pick a prime number $p > N$.
    \item Sample random $\bm{\mu}^\ast \in \mathbb{F}_p^\ell$, $\mathbf{x}^\ast \in \mathbb{F}_p^n$, and truncated up to order $N + h + 1$ power series $\mathbf{u}^\ast \in (\mathbb{F}_p[\![t]\!])^s$.
    \item Compute the truncated power series solution for~$\Sigma$ in $\mathbb{F}_p[\![t]\!]$ up to order $N + h + 1$ with the parameter value, initial conditions, and inputs being $\bm{\mu}^\ast$, $\mathbf{x}^{\ast}$, and $\mathbf{u}^\ast$, repectively (using the algorithm from~\cite{solutionsSODA}).
    If division by zero occurs, go to~\ref{step:sample_wrnsk} and choose a larger prime.
    \item\label{step:evaluation} Denote the $\mathbf{y}$-components of the computed solution by $\mathbf{y}^\ast$.
    Compute $g_1 = z_1(\mathbf{y}^\ast, \mathbf{u}^\ast), \ldots, g_N = z_N(\mathbf{y}^\ast, \mathbf{u}^\ast)$ (see Remark~\ref{rem:massive_eval}).
    \item Form an $N \times N$-matrix $M$ with the $i$-th column being the first $N$ Taylor coefficients in $g_i$ for $i = 1, \ldots, N$.
    \item If $M$ is nonsingular, return \texttt{Single-experiment}, otherwise return \texttt{Not sure}.
\end{enumerate}
\end{algorithm}

\begin{lemma}
    Algorithm~\ref{alg:wronsk} is correct.
\end{lemma}

\begin{proof}
    Let $\widetilde{M}$ be the matrix obtained from $M$ by multiplying the $j$-th row by $(j - 1)!$ for every $j = 1, \ldots, N$.
    Since $p > N$, $\det M \neq 0 \iff \det \widetilde{M} \neq 0$.
    We observe that $\widetilde{M}$ is the evaluation of the matrix $\wronsk(g_1, \ldots, g_N)$ at $t = 0$.
    If the evaluation is nonsingular, then the Wronskian as well as the Wronskian $\wronsk(z_1, \ldots, z_N) \pmod{I_\Sigma}$ are nonsingular, too.
    Therefore, the output value \texttt{Single-experiment} in this case is justified by Lemma~\ref{lem:wronsk}.
\end{proof}

\subsection{Algorithm for assessing identifiability}\label{sec:identifiability_algo}

\begin{algorithm}[H]
\caption{Assessing identifiability}\label{alg:main}
\begin{description}[itemsep=0pt]
\item[Input ] 
\begin{itemize}
    
    \item an ODE model~$\Sigma$ as in~\eqref{eq:ODEmodel} over $\mathbb{Q}$;
    \item a rational function $h(\bm{\mu}) \in \Q(\bm{\mu})$;
    \item a real number $0 < \varepsilon < 1$.
\end{itemize} 

\item[Output ]
\begin{itemize}
  \item One of
  \begin{itemize}
      \item \texttt{NO} if $h(\bm{\mu})$ is not multi-experiment identifiable;
      \item \texttt{Locally} if $h(\bm{\mu})$ is locally but not globally mutli-experiment identifiable;
      \item \texttt{Globally} if $h(\bm{\mu})$ is globally mutli-experiment identifiable
  \end{itemize}
  \item \texttt{True} if the result is also valid for single-experiment identifiability, and \texttt{False} if such a conclusion could not be made by the algorithm.
\end{itemize}
The result is guaranteed to be correct with probability at least $1 - \varepsilon$.
\end{description}

\begin{enumerate}[label = \textbf{(Step~\arabic*)}, leftmargin=*, align=left, labelsep=2pt, itemsep=0pt]
    \item\label{step:local} \emph{Assess local identifiability}
    \begin{enumerate}
        \item Use~\cite[Algorithm~1]{num_exp} (with probability of error at most $\varepsilon / 4$) to compute $r$ such that SE-idenitfiability in $\Sigma_r$ implies ME-identifiability.
        \item Use the algorithm from~\cite{Sedoglavic} (with probability of error at most $\varepsilon / 4$) to check whether $h(\bm{\mu})$ is $\Sigma_r$ locally-identifiable.
        If the answer is no, return (\texttt{NO}, \texttt{True}) if $r = 1$ or (\texttt{NO}, \texttt{False}) if $r > 1$.
    \end{enumerate}
    
    \item\label{step:io_proj} \emph{Compute input-output projections}
    \begin{enumerate}
        \item Apply Algorithm~\ref{alg:project_ode} to obtain the parametric profile $\mathbf{h}$ and the projections $\mathbf{f}$ for $I_\Sigma \cap \C(\bar{\mu})[\mathbf{y}^{(\infty)}, \mathbf{u}^{(\infty)}]$.
        \item Apply Algorithm~\ref{alg:field_def} to the projections $\mathbf{f}$ with the profile $\mathbf{h}$ and the weak membership test from Algorithm~\ref{alg:weak_membership} to obtain the polynomial $f$.
    \end{enumerate}
    
    \item\label{step:me_eq_se} \emph{SE-identifiable = ME-identifiable?}
    Apply Algorithm~\ref{alg:wronsk} to $\Sigma$ and each of $f, \mathbf{f}$ (after the normalization so that at least one coefficient is equal to one).
    If any of the results will be \texttt{Not sure}, set $F$ to be \texttt{False}, otherwise set it to be \texttt{True}.
    
    \item\label{step:check} \emph{Assess global identifiability}
    Let $c_1, \ldots, c_N$ be the coefficients of $f, \mathbf{f}$ after the normalization.
    Use Algorithm~\ref{alg:field_membership} (with probability of error at most $\varepsilon / 2$) to check whether $h \in \Q(c_1, \ldots, c_N)$.
    If the result is \texttt{YES}, return (\texttt{Globally}, F), otherwise return (\texttt{Locally}, F).
\end{enumerate}
\end{algorithm}

\begin{remark}[Simultaneous evaluation]\label{rem:massive_eval}
    The most computationally demanding step of Algorithm~\ref{alg:wronsk} is typically~\ref{step:evaluation}.
    This is because $N$ is often of order of hundreds or thousands, and $\mathbf{y}^\ast$ and $\mathbf{u}^\ast$ are truncated power series of order more than $N$, so performing many arithmetic operations with them (especially, multiplication) may take long time.
    
    The problem can be formulated in a general form as follows: given $N$ polynomials $q_1, \ldots, q_N \in R[x_1, \ldots, x_r]$ over ring $R$ and a tuple $\mathbf{a} \in R^r$, evaluate $q_1(\mathbf{a}), \ldots, q_N(\mathbf{a})$ in a way that would make the number of multiplications in $R$ smaller.
    We did this by first evaluating all the monomials appearing in $q_1, \ldots, q_N$ at $\mathbf{a}$, and then taking linear combinations of these results.
    For evaluating the monomials, we iterate through them in the ascending order with respect to degree and maintain a trie with the exponent vectors of already evaluated monomials.
    For each next monomial $m$, we find a monomial $m_0$ in the trie of largest degree such that $m_0 \mid m$ (this can be done in the linear time in the size of the trie), and reduce evaluation of $m$ to the evaluation of $m / m_0$.
    This approach leads to substantial speedup (e.g., 5 min. instead of 10 h. for Example~\ref{ex:pharm}).
\end{remark}


\begin{proposition}[Proof of Theorem~\ref{thm:ident}]
     The output of Algorithm~\ref{alg:main} is correct with probability at least $1 - \varepsilon$.
\end{proposition}

\begin{proof}
    If the function $h(\bm{\mu})$ is not identifiable, this will be detected at~\ref{step:local}.
    Moreover, this will be done correctly with probability at least $1 - 2 \varepsilon /4 = 1 - \varepsilon / 2$.
    
    If the function $h(\bm{\mu})$ is locally identifiable, then the cases ``globally'' and ``locally not globally'' will be distinguished by~\ref{step:io_proj} and~\ref{step:check}.
    Summing up the probability of errors, we see that the result will be correct with probability at least $1 - \varepsilon$.
    Finally, the returned value of $F$ will satisfy the specification due to the specification of Algorithm~\ref{alg:wronsk}.
\end{proof}


\section{Implementation and performance}\label{sec:performance}

We have implemented Algorithm~\ref{alg:main} (and all the algorithms it relies on) in {\sc Julia} language as a part of {\sc StructuralIdentifiability} package.
The package is publicly available at~\url{https://github.com/SciML/StructuralIdentifiability.jl} as a part of the SciML ecosystem\footnote{\url{https://sciml.ai/}}.
We use the symbolic computation libraries {\sc Nemo}~\cite{Nemo}, {\sc AbstractAlgebra.jl}, {\sc Singular.jl}, and {\sc GroebnerBases.jl}.
In this section we will demonstrate the performance of our implementation (version 0.2) on a set of benchmark problems.
The problems and their parameters are summarized in Table~\ref{tab:used_benchmarks} below and the explicit equations for most of them are collected in the Appendix.
All the timing in this paper are CPU times (not elapsed) measured on a laptop with Mac OS, 16 Gb RAM, and 4 cores 1.60GHz each.
The benchmark models in the formats of all the used software tools are available at~\url{https://github.com/SciML/StructuralIdentifiability.jl/tree/master/benchmarking}.

The rest of the section is structured as follows. 
In Section~\ref{sec:our_algo} we report and discuss the runtimes of the main steps of the algorithms on the benchmark problems.
Section~\ref{sec:comp_ident} contains the comparison of the performance of our implementation with several state-of-the art software tools for assessing global structural parameter identifiability, DAISY~\cite{DAISY}, SIAN~\cite{SIAN}, COMBOS~\cite{COMBOS}, and GenSSI~\cite{LFCBBCH}.
Finally, in Section~\ref{sec:comp_elim}, we compare the performance of the elimination algorithm we use (Algorithm~\ref{alg:project_ode}) for ODE models with general-purpose differential elimination algorithms (Rosenfeld-Gr\"obner and Differential Thomas from {\sc Maple}).

\begin{table}[H]
    \centering
    \resizebox{\textwidth}{!}{%
    \begin{tabular}{|l|c|c|c|c|c|}
    \hline
        \textbf{Name} & \textbf{Equations} & \textbf{states} & \textbf{params} & \textbf{outputs} & \textbf{inputs}\\
    \hline
    \hline
        SIWR model (Example~\ref{ex:SIWR}) & \eqref{bench:SIWR_main} \& \eqref{bench:SIWR_one} & $4$ & $7$ & $1$ & $0$ \\
    \hline
        SIWR model - 2 (Example~\ref{ex:SIWR}) & \eqref{bench:SIWR_main} \& \eqref{bench:SIWR_two} & $4$ & $7$ & $2$ & $0$ \\
    \hline
        Pharmacokinetics (Example~\ref{ex:pharm}) & \eqref{eq:Pharm} & $4$ & $7$ & $1$ & $0$ \\
    \hline
        MAPK pathway - 1 (Example~\ref{ex:mapk}) & \eqref{eq:MAPK_main} \& \eqref{eq:MAPK_out6} & $12$ & $22$ & $6$ & $0$ \\
    \hline
        MAPK pathway - 2 (Example~\ref{ex:mapk}) & \eqref{eq:MAPK_main} \& \eqref{eq:MAPK_out5} & $12$ & $22$ & $5$ & $0$ \\
    \hline
        MAPK pathway - 3 (Example~\ref{ex:mapk}) & \eqref{eq:MAPK_main} \& \eqref{eq:MAPK_out5bis} & $12$ & $22$ & $5$ & $0$ \\
    \hline
        SEAIJRC model (Example~\ref{ex:seaijrc}) & \eqref{eq:SEAIJRC} & $6$ & $8$ & $2$ & $0$ \\
    \hline
        Goodwin oscillator (Example~\ref{ex:Goodwin}) & \eqref{eq:Goodwin} & $4$ & $7$ & $1$ & $0$ \\
    \hline
        Akt pathway (Example~\ref{ex:Akt}) & \eqref{eq:Akt} \& \eqref{eq:Akt_out} & $9$ & $16$ & $3$ & $1$ \\
    \hline
        NF$\kappa$B (\cite[Section B.5]{SIAN}) & \cite[Section B.5]{SIAN} & $15$ & $13$ & $6$ & $1$ \\
    \hline
        Mass-action (\cite[Section B.1]{SIAN}) & \cite[Section B.1]{SIAN} & $6$ & $6$ & $2$ & $0$ \\
    \hline
        SIRS w. forcing (\cite[Section B.3]{SIAN}) & \cite[Section B.1]{SIAN} & $5$ & $6$ & $2$ & $0$ \\
    \hline
        
    \end{tabular}}
    \caption{Summary of the used benchmark problems}
    \label{tab:used_benchmarks}
\end{table}


\subsection{Performance and its discussion} \label{sec:our_algo}

For each of the benchmark models, Table~\ref{tab:our_performance} below contains
\begin{itemize}
    \item the runtime of our implementation with a breakdown for the steps of Algorithms~\ref{alg:main};
    \item indication whether the computed projection-based representation was actually a characteristic set of the elimination ideal as described in Remark~\ref{rem:charsets};
    \item whether or not Algorithm~\ref{alg:wronsk} could conclude that the coefficients of the computed input-output equations are single-experiment identifiable.
\end{itemize}

\begin{table}[H]
    \centering
    \resizebox{\textwidth}{!}{%
    \begin{tabular}{|l|c|c|c|c|c|c|c|}
    \hline
        \multirow{2}{*}{\textbf{Name}} & \multicolumn{5}{c|}{\textbf{Runtimes (sec.)}} & \rot{\multirow{2}{*}{\textbf{Char. set?}}} \hspace{1mm} & \rot{\multirow{2}{*}{\textbf{SE=ME?}}} \hspace{1mm}\\
        \cline{2-6} 
        & \ref{step:local} & \ref{step:io_proj} & \ref{step:me_eq_se} & \ref{step:check} & \textbf{Total} & & \\
    \hline
    \hline
        SIWR model (Example~\ref{ex:SIWR}) & $0.1$ & $3$ & $9.5$ & $5.3$ & $18$ & \OK & \OK \\
    \hline
        SIWR model - 2 (Example~\ref{ex:SIWR}) & $0.1$ & $0.2$ & $0.2$ & $0.2$ & $0.7$ & \OK & \OK \\
    \hline
    Pharmacokinetics (Example~\ref{ex:pharm}) & $0.1$ & $18.4$ & $344.4$ & $43$ & $406$ & \OK & \OK \\
    \hline
        MAPK pathway - 1 (Example~\ref{ex:mapk}) & $14.2$ & $2.5$ & $11.4$ & $11.4$ & $39.5$ & \OK & \OK \\
    \hline
        MAPK pathway - 2 (Example~\ref{ex:mapk}) & $4.9$ & $27.6$ & $14.7$ & $11$ & $58$ & \OK & \OK \\
    \hline
        MAPK pathway - 3 (Example~\ref{ex:mapk}) & $32.8$  & $397$ & $226$ & $429$ & $1084$ & \OK & \NOK \\
    \hline
        SEAIJRC model (Example~\ref{ex:seaijrc}) & $0.1$ & $28.6$ & $78$ & $25.2$ & $131.3$ & \OK & \OK \\
    \hline
        Goodwin oscillator (Example~\ref{ex:Goodwin}) & $<0.1$ & $<0.1$ & $<0.1$ & $<0.1$ & $0.2$ & \OK & \OK \\
    \hline
        Akt pathway (Example~\ref{ex:Akt}) & $1.5$ & $0.2$ & $2.3$ & $1$ & $5$ & \OK & \OK \\
    \hline
        NF$\kappa$B (\cite[Section B.5]{SIAN}) & $2$ & $> 5 \text{h.}$ & N/A & N/A & $> 5\text{h.}$ & N/A & N/A \\
    \hline
        Mass-action (\cite[Section B.1]{SIAN}) & $<0.1$ & $<0.1$ & $0.3$ & $<0.1$ & $0.5$ & \OK & \OK \\
    \hline
        SIRS w. forcing (\cite[Section B.3]{SIAN}) & $<0.1$ & $1$ & $28.2$ & $1$ & $30.3$ & \OK & \OK \\
    \hline
        
    \end{tabular}}
    \caption{Details on the performance of the implemented algorithm}
    \label{tab:our_performance}
\end{table}

We would like to make the following observations:
\begin{enumerate}
    \item For the majority of the models, computing input-output equations (i.e., \ref{step:io_proj}) is not the most time-consuming step (unlike, for example, DAISY~\cite{DAISY}).
    \item In all the benchmark models the computed input-output projections form a characteristic set of the elimination ideal as described in Remark~\ref{rem:charsets}.
    \item In all but one case (in Example~\ref{ex:mapk}), the algorithm was able to conclude that the coefficients of the input-output equations are SE-identifiable.
    To the best of our knowledge, this is the first case of a practically relevant model without constant states exhibiting such a behavior.
    Using SIAN~\cite{SIAN} we have verified that all the parameters are in fact SE-identifiable.
    
\end{enumerate}


\subsection{Comparison with identifiability software} \label{sec:comp_ident}
In this section we compare the performance of our implementation of Algorithm~\ref{alg:main} with popular software packages for assessing local and global structural parameter identifiability:
\begin{itemize}
    \item {\sc SIAN}: software writte in {\sc Maple}~\cite{SIAN}.
    The algorithm uses an improved version of the Taylor series approach (for details, see~\cite{HOPY2020}) and can assess single-experiment identifiability of both parameters and initial conditions.
    The algorithm is a randomized Monte-Carlo algorithm, and we were running it with the default probability $99\%$.
    We used version 1.5 of SIAN with {\sc Maple 2021}.
    The reported runtimes were measured using the \texttt{CPUTime} function in {\sc Maple}.
    
    \item {\sc DAISY}: package written in {\sc Reduce}~\cite{DAISY}.
    It uses the approach via input-output equations but does not check that their coefficients are SE-identifiable~\cite[Example~2.14]{HOPY2020}.
    The algorithm is randomized but no probability bound is provided.
    We used DAISY 2.0 and {\sc Reduce rev. 4859}.
    The reported runtimes were measured using the \texttt{Showtime} function in {\sc Reduce}.
    
    \item {\sc COMBOS}: web-application described in~\cite{COMBOS}.
    Similarly to DAISY, it uses input-output equations but does not check that their coefficients are SE-identifiable and does not provide a probability bound.
    
    \item {\sc GenSSI 2.0}: package written in {\sc Matlab}~\cite{LFCBBCH}. 
    It uses the generating series approach.
\end{itemize}

Table~\ref{tab:comparison_identifiability} contains the timings of our implementation, DAISY, and SIAN on the set of benchmarks.
We did not include GenSSI and COMBOS into the table because, for all the example, the computation either did not finish in 5 hours or returned an error (\emph{``Warning: Unable to find explicit solution''} for GenSSI and \emph{``Model may have been entered incorrectly or cannot be solved with COMBOS algorithms''} for COMBOS).

From the table, one can see that the performance of our algorithm compares favorably to the state-of-the-art software and performs identifiability analysis for systems that were previously out of reach.

\subsection{Comparison with general-purpose differential elimination algorithms} \label{sec:comp_elim}
In this section we compare the performance of Algorithm~\ref{alg:project_ode} (which is used as a subroutine in the main Algorithm~\ref{alg:main}) for computing a PB-representation of the ideal of input-output relations of an ODE system to two state-of-the-art general purpose {\sc Maple} packages for differential elimination: {\sc DifferentialAlgebra} (built on top of the {\sc BLAD} library) and {\sc DifferentialThomas} representing the result of elimination via a characteristic set and simple subsystems, respectively.
Note that, for all the considered benchmarks, the PB-representation is actually a characteristic set (see Remark~\ref{rem:charsets}).
We used {\sc Maple 2021} and all the reported runtimes were measured using the \texttt{CPUTime} function.
Table~\ref{tab:comparison_elimination} contains the runtimes.
One can see that many models which are out of reach for general-purpose algorithms can be tackled by our dedicated algorithm.

\begin{table}[H]
    \centering
    \begin{tabular}{|l|r|r|r|r|}
    \hline
        \textbf{Model} & {\sc DAISY} & {\sc SIAN} & Algorithm~\ref{alg:main} (our) \\
    \hline
    \hline
        SIWR model (Example~\ref{ex:SIWR}) & OOM & $> 5\,\text{h.}$ & $18\,\text{s.}$\\
    \hline
        SIWR model - 2 (Example~\ref{ex:SIWR}) & OOM & $213\,\text{s.}$ & $0.7\,\text{s.}$\\
    \hline
    Pharmacokinetics (Example~\ref{ex:pharm}) & $> 5\,\text{h.}$ & $> 5\,\text{h.}$ & $406\,\text{s.}$ \\
    \hline
        MAPK pathway - 1 (Example~\ref{ex:mapk}) & OOM & $31\,\text{s.}$ & $39.5\,\text{s.}$ \\
    \hline
        MAPK pathway - 2 (Example~\ref{ex:mapk}) & $> 5\,\text{h.}$ & $> 5\,\text{h.}$ & $58\,\text{s.}$ \\
    \hline
        MAPK pathway - 3 (Example~\ref{ex:mapk}) & $> 5\,\text{h.}$ & $35\,\text{s.}$ & $1084\,\text{s.}$ \\
    \hline
        SEAIJRC model (Example~\ref{ex:seaijrc}) & OOM & $> 5\,\text{h.}$ & $131.3\,\text{s.}$\\
    \hline
        Goodwin oscillator (Example~\ref{ex:Goodwin}) & $18\,\text{s.}^*$ & $4920\,\text{s.}$ & $0.2\,\text{s.}$ \\
    \hline
        Akt pathway (Example~\ref{ex:Akt}) & $182\,\text{s.}$ & $28\,\text{s.}$ & $5\,\text{s.}$ \\
    \hline
        NF$\kappa$B (\cite[Section B.5]{SIAN}) & $> 5\,\text{h.}$ & $2018\,\text{s.}$ & $> 5\,\text{h.}$\\
    \hline
        Mass-action (\cite[Section B.1]{SIAN}) & $> 5\,\text{h.}$ & $3\,\text{s.}$ & $0.5\,\text{s.}$\\
    \hline
        SIRS w. forcing (\cite[Section B.3]{SIAN}) & OOM & $5\,\text{s.}$ & $30.3\,\text{s.}$\\
    \hline
        
    \end{tabular}
    \caption{Comparison with other identifiability software}
    {\small OOM: ``out of memory'';\hspace{3mm}
    $\ast$: the output is different from the one by {\sc SIAN} and Algorithm~\ref{alg:main} }
    \label{tab:comparison_identifiability}
\end{table}


\begin{table}[H]
    \centering
    \resizebox{\textwidth}{!}{%
    \begin{tabular}{|l|r|r|r|}
    \hline
        \textbf{Model} & {\sc DiffAlgebra} & {\sc DiffThomas} & Algorithm~\ref{alg:project_ode} (our) \\
    \hline
    \hline
        SIWR model (Example~\ref{ex:SIWR}) & $> 5\,\text{h.}$ & $> 5\,\text{h.}$ & $3\,\text{s.}$ \\
    \hline
        SIWR model - 2 (Example~\ref{ex:SIWR}) & $> 5\,\text{h.}$ & $> 5\,\text{h.}$ & $0.2\,\text{s.}$ \\
    \hline
    Pharmacokinetics (Example~\ref{ex:pharm}) & OOM & $> 5\,\text{h.}$ & $18.4\,\text{s.}$ \\
    \hline
        MAPK pathway - 1 (Example~\ref{ex:mapk}) & $13\,\text{s.}$ & $7\,\text{s.}$ & $2.5\,\text{s.}$ \\
    \hline
        MAPK pathway - 2 (Example~\ref{ex:mapk}) & $> 5\,\text{h.}$ & $> 5\,\text{h.}$ & $27.6\,\text{s.}$ \\
    \hline
        MAPK pathway - 3 (Example~\ref{ex:mapk}) & $> 5\,\text{h.}$ & $> 5\,\text{h.}$ & $397\,\text{s.}$ \\
    \hline
        SEAIJRC model (Example~\ref{ex:seaijrc}) & $> 5\,\text{h.}$ & $> 5\,\text{h.}$ & $28.6\,\text{s.}$ \\
    \hline
        Goodwin oscillator (Example~\ref{ex:Goodwin}) & $0.2\,\text{s.}$ &  $0.4\,\text{s.}$ & $< 0.1\,\text{s.}$ \\
    \hline
        Akt pathway (Example~\ref{ex:Akt}) & $0.2\,\text{s.}$ & $> 5\,\text{h.}$ & $0.2\,\text{s.}$ \\
    \hline
        NF$\kappa$B (\cite[Section B.5]{SIAN}) & $> 5\,\text{h.}$ & $> 5\,\text{h.}$ & $> 5\,\text{h.}$ \\
    \hline
        Mass-action (\cite[Section B.1]{SIAN}) & $4.7\,\text{s.}$ & $> 5\,\text{h.}$ & $<0.1\,\text{s.}$\\
    \hline
        SIRS w. forcing (\cite[Section B.3]{SIAN}) & $> 5\,\text{h.}$  & $> 5\,\text{h.}$ & $1\,\text{s.}$\\
    \hline
        
    \end{tabular}}
    \caption{Comparison with general purpose libraries for differential elimination}
    OOM: ``out of memory''
    \label{tab:comparison_elimination}
\end{table}


\subsection*{Acknowledgements}

The authors are grateful to Emilie Dufresne, Hoon Hong and Joris van der Hoeven for helpful discussions.
The authors are grateful to the referees for careful reading, helpful comments and suggestions including streamlining the proofs of Proposition~\ref{prop:top_dim_comp} and Lemma~\ref{lem:repres}.
GP was partially supported by NSF grants DMS-1853482, DMS-1760448, and  DMS-1853650, by the Paris Ile-de-France region, and by the MIMOSA project funded by AAP INS2I CNRS.
HAH gratefully acknowledges a Royal Society University Research Fellowship UF150238 and RGF\textbackslash EA\textbackslash 201074 and partial support by EPSRC EP/R005125/1 and EP/T001968/1.


\bibliographystyle{siamplain}
\bibliography{references}

\begin{thebibliography}{10}

\bibitem{diff_Thomas}
{\sc T.~B\"{a}chler, V.~Gerdt, M.~Lange-Hegermann, and D.~Robertz}, {\em Thomas
  decomposition of algebraic and differential systems}, in Computer Algebra in
  Scientific Computing, 2010, pp.~31--54,
  \url{https://doi.org/10.1007/978-3-642-15274-0_4}.

\bibitem{BHM19}
{\sc D.~J. Bates, J.~D. Hauenstein, and N.~Meshkat}, {\em Identifiability and
  numerical algebraic geometry}, {PLOS} {ONE}, 14 (2019), p.~e0226299,
  \url{https://doi.org/10.1371/journal.pone.0226299}.

\bibitem{bellman-astrom-70}
{\sc R.~Bellman and K.~{\r{A}str\"{o}m}}, {\em On structural identifiability},
  Mathematical Biosciences, 7 (1970), pp.~329--339,
  \url{http://dx.doi.org/10.1016/0025-5564(70)90132-X}.

\bibitem{DAISY}
{\sc G.~Bellu, M.~P. Saccomani, S.~Audoly, and L.~D'Angi{\`o}}, {\em {DAISY}: A
  new software tool to test global identifiability of biological and
  physiological systems}, Computer methods and programs in biomedicine, 88
  (2007), pp.~52--61, \url{https://doi.org/10.1016/j.cmpb.2007.07.002}.

\bibitem{solutionsSODA}
{\sc A.~Bostan, F.~Chyzak, F.~Ollivier, B.~Salvy, E.~Schost, and
  A.~Sedoglavic}, {\em Fast computation of power series solutions of systems of
  differential equations}, in Proceedings of the Annual ACM-SIAM Symposium on
  Discrete Algorithms, 2007, p.~1012–1021.

\bibitem{blad}
{\sc F.~Boulier}, {\em {BLAD}: {B}ibliothèques {L}illoises d'{A}lgèbre
  {D}ifférentielle},
  \url{https://pro.univ-lille.fr/francois-boulier/logiciels/blad/}.

\bibitem{Boulier2}
{\sc F.~Boulier, D.~Lazard, F.~Ollivier, and M.~Petitot}, {\em Computing
  representations for radicals of finitely generated differential ideals},
  Applicable Algebra in Engineering, Communication and Computing, 20 (2009),
  pp.~73--121, \url{https://doi.org/10.1007/s00200-009-0091-7}.

\bibitem{Boulier2010}
{\sc F.~Boulier, F.~Lemaire, and M.~M. Maza}, {\em Computing differential
  characteristic sets by change of ordering}, Journal of Symbolic Computation,
  45 (2010), pp.~124--149, \url{https://doi.org/10.1016/j.jsc.2009.09.004}.

\bibitem{repeated}
{\sc L.~Bus\'e and B.~Mourrain}, {\em Explicit factors of some iterated
  resultants and discriminants}, Mathematics of Computation, 78 (2009),
  pp.~345--386, \url{http://www.jstor.org/stable/40234778}.

\bibitem{CD17}
{\sc A.~Castillo and R.~Dietmann}, {\em On {H}ilbert’s irreducibility
  theorem}, Acta Arithmetica, 180 (2017), pp.~1--14,
  \url{https://doi.org/10.4064/aa8380-2-2017}.

\bibitem{comparison}
{\sc O.-T. Chis, J.~R. Banga, and E.~Balsa-Canto}, {\em Structural
  identifiability of systems biology models: A critical comparison of methods},
  {PLoS} {ONE}, 6 (2011), p.~e27755,
  \url{https://doi.org/10.1371/journal.pone.0027755}.

\bibitem{Cohen}
{\sc S.~D. Cohen}, {\em The distribution of {G}alois groups and {H}ilbert's
  irreducibility theorem}, Proceedings of the London Mathematical Society,
  s3-43 (1981), pp.~227--250, \url{https://doi.org/10.1112/plms/s3-43.2.227}.

\bibitem{COLLART1997}
{\sc S.~Collart, M.~Kalkbrenner, and D.~Mall}, {\em Converting bases with the
  gr\"{o}bner walk}, Journal of Symbolic Computation, 24 (1997), pp.~465--469,
  \url{https://doi.org/10.1006/jsco.1996.0145}.

\bibitem{Conte2007}
{\sc G.~Conte, C.~H. Moog, and A.~M. Perdon}, {\em Algebraic Methods for
  Nonlinear Control Systems}, Springer London, 2007,
  \url{https://doi.org/10.1007/978-1-84628-595-0}.

\bibitem{CLO}
{\sc D.~Cox, J.~Little, and D.~O'Shea}, {\em Using Algebraic Geometry},
  Springer New York, NY, 1998, \url{https://doi.org/10.1007/978-1-4757-6911-1}.

\bibitem{Dahan2008}
{\sc X.~Dahan, X.~Jin, M.~M. Maza, and {\'{E}}.~Schost}, {\em Change of order
  for regular chains in positive dimension}, Theoretical Computer Science, 392
  (2008), pp.~37--65, \url{https://doi.org/10.1016/j.tcs.2007.10.003}.

\bibitem{DS04}
{\sc X.~Dahan and E.~Schost}, {\em Sharp estimates for triangular sets}, in
  Proceedings of the 2004 International Symposium on Symbolic and Algebraic
  Computation, ISSAC '04, ACM, 2004, pp.~103--110,
  \url{https://doi.org/10.1145/1005285.1005302}.

\bibitem{Decker}
{\sc W.~Decker, G.-M. Greuel, and G.~Pfister}, {\em Primary decomposition:
  Algorithms and comparisons}, in Algorithmic Algebra and Number Theory, B.~H.
  Matzat, G.-M. Greuel, and G.~Hiss, eds., Springer Berlin Heidelberg, 1999,
  pp.~187--220, \url{https://doi.org/10.1007/978-3-642-59932-3_10}.

\bibitem{Pharm}
{\sc S.~Demignot and D.~Domurado}, {\em Effect of prosthetic sugar groups on
  the pharmacokinetics of glucose-oxidase}, Drug Design and Delivery, 1 (1987),
  pp.~333--348.

\bibitem{DenisVidal2003}
{\sc L.~Denis-Vidal, G.~Joly-Blanchard, and C.~Noiret}, {\em System
  identifiability (symbolic computation) and parameter estimation (numerical
  computation)}, Numerical Algorithms, 34 (2003), pp.~283--292,
  \url{https://doi.org/10.1023/b:numa.0000005366.05704.88}.

\bibitem{DVJBNP01}
{\sc L.~Denis-Vidal, G.~Joly-Blanchard, C.~Noiret, and M.~Petitot}, {\em An
  algorithm to test identifiability of non-linear systems}, IFAC Proceedings
  Volumes, 34 (2001), pp.~197--201,
  \url{https://doi.org/10.1016/S1474-6670(17)35173-X}.

\bibitem{eisenbud}
{\sc D.~Eisenbud}, {\em Commutative Algebra with a View Toward Algebraic
  Geometry}, Springer-Verlag New York, 1995,
  \url{https://doi.org/10.1007/978-1-4612-5350-1}.

\bibitem{Nemo}
{\sc C.~Fieker, W.~Hart, T.~Hofmann, and F.~Johansson}, {\em Nemo/{H}ecke:
  Computer algebra and number theory packages for the {J}ulia programming
  language}, in Proceedings of the 2017 ACM on International Symposium on
  Symbolic and Algebraic Computation, ISSAC '17, ACM, 2017, pp.~157--164,
  \url{http://doi.acm.org/10.1145/3087604.3087611}.

\bibitem{Fujita2010}
{\sc K.~A. Fujita, Y.~Toyoshima, S.~Uda, Y.~i.~Ozaki, H.~Kubota, and
  S.~Kuroda}, {\em Decoupling of receptor and downstream signals in the akt
  pathway by its low-pass filter characteristics}, Science Signaling, 3 (2010),
  pp.~ra56--ra56, \url{https://doi.org/10.1126/scisignal.2000810}.

\bibitem{GLS}
{\sc M.~Giusti, G.~Lecerf, and B.~Salvy}, {\em A {G}r\"{o}bner free alternative
  for polynomial system solving}, Journal of Complexity, 17 (2001),
  pp.~154--211, \url{https://doi.org/10.1006/jcom.2000.0571}.

\bibitem{Glumineau1996}
{\sc A.~Glumineau, C.~Moog, and F.~Plestan}, {\em New algebraic-geometric
  conditions for the linearization by input-output injection}, {IEEE}
  Transactions on Automatic Control, 41 (1996), pp.~598--603,
  \url{https://doi.org/10.1109/9.489283}.

\bibitem{Golubitsky2009}
{\sc O.~Golubitsky, M.~Kondratieva, and A.~Ovchinnikov}, {\em Algebraic
  transformation of differential characteristic decompositions from one ranking
  to another}, Journal of Symbolic Computation, 44 (2009), pp.~333--357,
  \url{https://doi.org/10.1016/j.jsc.2008.07.002}.

\bibitem{Goodwin}
{\sc B.~C. Goodwin}, {\em Oscillatory behavior in enzymatic control processes},
  Advances in Enzyme Regulation, 3 (1965), pp.~425--437,
  \url{https://doi.org/10.1016/0065-2571(65)90067-1}.

\bibitem{selection}
{\sc H.~A. Harrington, K.~L. Ho, and N.~Meshkat}, {\em A parameter-free model
  comparison test using differential algebra}, Complexity, 2019 (2019),
  pp.~1--15, \url{https://doi.org/10.1155/2019/6041981}.

\bibitem{Heintz1983}
{\sc J.~Heintz}, {\em Definability and fast quantifier elimination in
  algebraically closed fields}, Theoretical Computer Science, 24 (1983),
  pp.~239--277, \url{https://doi.org/10.1016/0304-3975(83)90002-6}.

\bibitem{SIAN}
{\sc H.~Hong, A.~Ovchinnikov, G.~Pogudin, and C.~Yap}, {\em {SIAN}: software
  for structural identifiability analysis of {ODE} models}, Bioinformatics, 35
  (2019), pp.~2873--2874, \url{https://doi.org/10.1093/bioinformatics/bty1069}.

\bibitem{HOPY2020}
{\sc H.~Hong, A.~Ovchinnikov, G.~Pogudin, and C.~Yap}, {\em Global
  identifiability of differential models}, Communications on Pure and Applied
  Mathematics, 73 (2020), pp.~1831--1879,
  \url{https://doi.org/10.1002/cpa.21921}.

\bibitem{Hubert2003}
{\sc E.~Hubert}, {\em Notes on triangular sets and triangulation-decomposition
  algorithms {I}: Polynomial systems}, in Lecture Notes in Computer Science,
  Springer Berlin Heidelberg, 2003, pp.~1--39,
  \url{https://doi.org/10.1007/3-540-45084-x_1}.

\bibitem{Hubert2003b}
{\sc E.~Hubert}, {\em Notes on triangular sets and triangulation-decomposition
  algorithms {II}: Differential systems}, in Lecture Notes in Computer Science,
  Springer Berlin Heidelberg, 2003, pp.~40--87,
  \url{https://doi.org/10.1007/3-540-45084-x_2}.

\bibitem{webapp}
{\sc I.~Ilmer, A.~Ovchinnikov, and G.~Pogudin}, {\em Web-based structural
  identifiability analyzer}, in Computational Methods in Systems Biology, 2021,
  pp.~254--265, \url{https://doi.org/10.1007/978-3-030-85633-5_17}.

\bibitem{Jiafan2009}
{\sc Z.~Jiafan}, {\em Nonlinear systems fault diagnosis with differential
  elimination}, in 2009 International Conference on Computational Intelligence
  and Natural Computing, 2009, \url{https://doi.org/10.1109/cinc.2009.38}.

\bibitem{Komatsu2020}
{\sc M.~Komatsu, T.~Yaguchi, and K.~Nakajima}, {\em Algebraic approach towards
  the exploitation of {``}softness{''}: the input-output equation for
  morphological computation}, The International Journal of Robotics Research,
  (2020), \url{https://doi.org/10.1177/0278364920912298}.

\bibitem{Cholera}
{\sc E.~C. Lee, M.~R. Kelly, B.~M. Ochocki, S.~M. Akinwumi, K.~E. Hamre, J.~H.
  Tien, and M.~C. Eisenberg}, {\em Model distinguishability and inference
  robustness in mechanisms of cholera transmission and loss of immunity},
  Journal of Theoretical Biology, 420 (2017), pp.~68--81,
  \url{http://dx.doi.org/10.1016/j.jtbi.2017.01.032}.

\bibitem{LFCBBCH}
{\sc T.~Ligon, F.~Fr{\"o}hlich, O.~T. Chi{\c{s}}, J.~Banga, E.~Balsa-Canto, and
  J.~Hasenauer}, {\em {GenSSI} 2.0: multi-experiment structural identifiability
  analysis of {SBML} models}, Bioinformatics, 34 (2018), pp.~1421--1423,
  \url{http://dx.doi.org/10.1093/bioinformatics/btx735}.

\bibitem{LG94}
{\sc L.~Ljung and T.~Glad}, {\em On global identifiability for arbitrary model
  parametrizations}, Automatica, 30 (1994), pp.~265--276,
  \url{https://doi.org/10.1016/0005-1098(94)90029-9}.

\bibitem{manrai2008geometry}
{\sc A.~K. Manrai and J.~Gunawardena}, {\em The geometry of multisite
  phosphorylation}, Biophysical journal, 95 (2008), pp.~5533--5543.

\bibitem{MBV20}
{\sc G.~Massonis, J.~R. Banga, and A.~F. Villaverde}, {\em Structural
  identifiability and observability of compartmental models of the {COVID}-19
  pandemic}, Annual Reviews in Control,  (2020),
  \url{https://doi.org/10.1016/j.arcontrol.2020.12.001}.

\bibitem{MED2009}
{\sc N.~Meshkat, M.~Eisenberg, and J.~DiStefano}, {\em An algorithm for finding
  globally identifiable parameter combinations of nonlinear {ODE} models using
  {G}r{\"o}bner bases}, Mathematical Biosciences, 222 (2009), pp.~61--72,
  \url{https://doi.org/10.1016/j.mbs.2009.08.010}.

\bibitem{COMBOS}
{\sc N.~Meshkat, C.~Kuo, and J.~DiStefano}, {\em On finding and using
  identifiable parameter combinations in nonlinear dynamic systems biology
  models and {COMBOS}: A novel web implementation}, PLoS ONE, 9 (2014),
  p.~e110261, \url{https://doi.org/10.1371/journal.pone.0110261}.

\bibitem{OllivierPhD}
{\sc F.~Ollivier}, {\em Le probl{\`e}me de l’identifiabilit{\'e} structurelle
  globale: approche th {\'e}orique, m{\'e}thodes effectives et bornes de
  complexit{\'e}}, PhD thesis, {\'E}cole polytechnique, 1990,
  \url{https://www.theses.fr/1990EPXX0009}.

\bibitem{num_exp}
{\sc A.~Ovchinnikov, A.~Pillay, G.~Pogudin, and T.~Scanlon}, {\em
  Multi-experiment parameter identifiability of {ODE}s and model theory},
  (2020), \url{https://arxiv.org/abs/2011.10868}.

\bibitem{allident}
{\sc A.~Ovchinnikov, A.~Pillay, G.~Pogudin, and T.~Scanlon}, {\em Computing all
  identifiable functions of parameters for {ODE} models}, Systems {\&} Control
  Letters, 157 (2021), p.~105030,
  \url{https://doi.org/10.1016/j.sysconle.2021.105030}.

\bibitem{OPT19}
{\sc A.~Ovchinnikov, G.~Pogudin, and P.~Thompson}, {\em Input-output equations
  and identifiability of linear {ODE} models}.
\newblock 2020, \url{https://arxiv.org/abs/1910.03960}.

\bibitem{ioaaecc}
{\sc A.~Ovchinnikov, G.~Pogudin, and P.~Thompson}, {\em Parameter
  identifiability and input-output equations}, Applicable Algebra in
  Engineering, Communication and Computing,  (2021),
  \url{https://doi.org/10.1007/s00200-021-00486-8}.

\bibitem{noether}
{\sc G.~Pogudin}, {\em A differential analog of the {N}oether normalization
  lemma}, International Mathematics Research Notices,  (2017),
  \url{dx.doi.org/10.1093/imrn/rnw275}.

\bibitem{Remien2021}
{\sc C.~H. Remien, M.~J. Eckwright, and B.~J. Ridenhour}, {\em Structural
  identifiability of the generalized {L}otka-{V}olterra model for microbiome
  studies}, Royal Society Open Science, 8 (2021), p.~201378,
  \url{https://doi.org/10.1098/rsos.201378}.

\bibitem{Ritt}
{\sc J.~F. Ritt}, {\em Differential Equations from the Algebraic Standpoint},
  Colloquium Publications, American Mathematical Society, 1932.

\bibitem{RC19}
{\sc K.~Roosa and G.~Chowell}, {\em Assessing parameter identifiability in
  compartmental dynamic models using a computational approach: application to
  infectious disease transmission models}, Theoretical Biology and Medical
  Modelling, 16 (2019), \url{https://doi.org/10.1186/s12976-018-0097-6}.

\bibitem{Sedoglavic}
{\sc A.~Sedoglavic}, {\em A probabilistic algorithm to test local algebraic
  observability in polynomial time}, Journal of Symbolic Computation, 33
  (2002), pp.~735--755, \url{https://doi.org/10.1006/jsco.2002.0532}.

\bibitem{Seidenberg}
{\sc A.~Seidenberg}, {\em An elimination theory for differential algebra},
  University of California publications in Mathematics, III (1956), pp.~31--66.

\bibitem{Sontag1998}
{\sc E.~D. Sontag}, {\em Mathematical Control Theory}, Springer New York, 1998,
  \url{https://doi.org/10.1007/978-1-4612-0577-7}.

\bibitem{Staroswiecki2001}
{\sc M.~Staroswiecki and G.~Comtet-Varga}, {\em Analytical redundancy relations
  for fault detection and isolation in algebraic dynamic systems}, Automatica,
  37 (2001), pp.~687--699, \url{https://doi.org/10.1016/s0005-1098(01)00005-x}.

\bibitem{Tuncer2021}
{\sc N.~Tuncer and M.~Martcheva}, {\em Determining reliable parameter estimates
  for within-host and within-vector models of zika virus}, Journal of
  Biological Dynamics, 15 (2021), pp.~430--454,
  \url{https://doi.org/10.1080/17513758.2021.1970261}.

\bibitem{vanderHoeven2002}
{\sc J.~van~der Hoeven}, {\em Relax, but don't be too lazy}, J. of Symbolic
  Computation, 34 (2002), pp.~479--542,
  \url{https://doi.org/10.1006/jsco.2002.0562}.

\bibitem{vanderHoeven2010}
{\sc J.~van~der Hoeven}, {\em Newton's method and {FFT} trading}, J. of
  Symbolic Computation, 45 (2010), pp.~857--878,
  \url{https://doi.org/10.1016/j.jsc.2010.03.005}.

\bibitem{Verdiere2005}
{\sc N.~Verdiere, L.~Denis-Vidal, G.~Joly-Blanchard, and D.~Domurado}, {\em
  Identifiability and estimation of pharmacokinetic parameters for the ligands
  of the macrophage {M}annose receptor}, International Journal of Applied
  Mathematics and Computer Science, 15 (2005), pp.~517--526,
  \url{http://eudml.org/doc/207763}.

\bibitem{WANG2002}
{\sc D.~Wang}, {\em {EPSILON}: A library of software tools for polynomial
  elimination}, in Mathematical Software, 2002,
  \url{https://doi.org/10.1142/9789812777171_0040}.

\bibitem{Wang1995}
{\sc Y.~Wang and E.~D. Sontag}, {\em Orders of input/output differential
  equations and state-space dimensions}, {SIAM} Journal on Control and
  Optimization, 33 (1995), pp.~1102--1126,
  \url{https://doi.org/10.1137/s0363012993246828}.

\bibitem{XiaMoog}
{\sc X.~Xia and C.~Moog}, {\em Identifiability of nonlinear systems with
  application to {HIV}/{AIDS} models}, {IEEE} Transactions on Automatic
  Control, 48 (2003), pp.~330--336,
  \url{https://doi.org/10.1109/tac.2002.808494}.

\bibitem{Yeung2020}
{\sc E.~Yeung, S.~McFann, L.~Marsh, E.~Dufresne, S.~Filippi, H.~A. Harrington,
  S.~Y. Shvartsman, and M.~W\"{u}hr}, {\em Inference of multisite
  phosphorylation rate constants and their modulation by pathogenic mutations},
  Current Biology, 30 (2020), pp.~877--882.e6,
  \url{https://doi.org/10.1016/j.cub.2019.12.052}.

\bibitem{Zippel}
{\sc R.~Zippel}, {\em Effective Polynomial Computation}, Springer, 1993,
  \url{http://dx.doi.org/10.1007/978-1-4615-3188-3}.

\end{thebibliography}


\section*{Appendix: Benchmark models used in Section~\ref{sec:performance}}

In the appendix, we collect some of the models used as benchmarks in the paper.
We did not include some relatively long models already explicitly formulated in the literature.

\begin{example}[SIWR model]\label{ex:SIWR}
  The following extension of the classical SIR model was proposed by Lee et al~\cite[Eq. (3)]{Cholera} to model Cholera:
  \begin{equation}\label{bench:SIWR_main}
  \begin{cases}
    \dot{s} & = \mu - \beta_I s i - \beta_W s w - \mu s + \alpha r, \\ 
    \dot{i} & = \beta_W s w + \beta_I s i - \gamma i - \mu i, \\ 
    \dot{w} & = \xi (i - w), \\ 
    \dot{r} & = \gamma i - \mu r - \alpha r,
   \end{cases}
  \end{equation}
  where the state variables $s$, $i$, and $r$ denote the fractions of the population that are
  susceptible, infectious, and recovered, respectively.  The variable $w$ represents the concentration of the bacteria in the environment.
  The following output was considered in the original paper by Lee et al~\cite{Cholera}
  \begin{equation}\label{bench:SIWR_one}
    y_1 = \kappa i.
  \end{equation}
  In~\cite{HOPY2020}, the following extended set of outputs was used
  \begin{equation}\label{bench:SIWR_two}
    y_1 = \kappa i \quad \text{ and }\quad y_2 = s + i + r.
  \end{equation}
\end{example}

\begin{example}[Pharmacokinetics]\label{ex:pharm}
This model was first presented by Demignot and Domurado in~\cite{Pharm}, and its identifiability was studied in~\cite[Case study 2]{comparison}:
  \begin{equation}\label{eq:Pharm}
  \begin{cases}
    \dot x_1 = a (x_2 - x_1) - \frac{k_a V_m x_1}{k_c k_a + k_c x_3 + k_a x_1},\\
    \dot x_2 = a (x_1 - x_2),\\
    \dot x_3 = b_1(x_4 - x_3) - \frac{k_cV_m x_3}{k_c k_a + k_c x_3 + k_a x_1},\\
    \dot x_4 = b_2(x_3 - x_4),\\
    y = x_1.
  \end{cases}  
  \end{equation}
  Its simplified version has been used in~\cite[Section B.6]{SIAN} and \cite[Example 6.4]{HOPY2020}.
\end{example}

\begin{example}[MAPK pathway]\label{ex:mapk}
  The following model describing mitogen-activated protein kinase (MAPK) was presented and analuzed by Manrai and Gunawardenahas~\cite{manrai2008geometry}
  \begin{equation}\label{eq:MAPK_main}
  \begin{cases}
      KS_{00}' = -a_{00} K \cdot S_{00} + b_{00} KS_{00} + \gamma_{0100} FS_{01} + \gamma_{1000} FS_{10} + \gamma_{1100} FS_{11},\\
      KS_{01}' = -a_{01} K\cdot S_{01} + b_{01}  KS_{01} + c_{0001} KS_{00} - \alpha_{01} F \cdot S_{01} + \beta_{01} FS_{01} + \gamma_{1101} FS_{11},\\
      KS_{10}' = -a_{10} K\cdot S_{10} + b_{10} KS_{10} + c_{0010} KS_{00} - \alpha_{10} F \cdot S_{10} + \beta_{10} FS_{10} + \gamma_{1110} FS_{11},\\
      FS_{01}' = -\alpha_{11} F\cdot S_{11} + \beta_{11} FS_{11} + c_{0111} KS_{01} + c_{1011} KS_{10} + c_{0011} KS_{00},\\
      FS_{10}' = a_{00} K\cdot S_{00} - (b_{00} + c_{0001} + c_{0010} + c_{0011}) KS_{00},\\
      FS_{11}' = a_{01} \cdot S_{01} - (b_{01} + c_{0111}) KS_{01},\\
      K' = a_{10} K\cdot S_{10} - (b_{10} + c_{1011}) KS_{10},\\
      F' = \alpha_{01} F\cdot S_{01} - (\beta_{01} + \gamma_{0100}) FS_{01},\\
      S_{00}' = \alpha_{10} F\cdot S_{10} - (\beta_{10} + \gamma_{1000}) FS_{10},\\
      S_{01}' = \alpha_{11} F\cdot S_{11} - (\beta_{11} + \gamma_{1101} + \gamma_{1110} + \gamma_{1100})  FS_{11},\\
      S_{10}' = -a_{00} K\cdot S_{00} + (b_{00} + c_{0001} + c_{0010} + c_{0011}) KS_{00} - a_{01} K\cdot S_{01} +\\ \hspace{10mm} (b_{01} + c_{0111}) KS_{01} - a_{10} K\cdot S_{10} + (b_{10} + c_{1011}) KS_{10},\\
      S_{11}' = -\alpha_{01} F\cdot  S_{01} + (\beta_{01} + \gamma_{0100}) FS_{01} - \alpha_{10} F\cdot S_{10} + (\beta_{10} + \gamma_{1000}) FS_{10} - \alpha_{11} F \cdot S_{11} + \\
      \hspace{10mm}(\beta_{11} + \gamma_{1101} + \gamma{1110} + \gamma{1100}) FS_{11}.
  \end{cases}
  \end{equation}
  We will consider three sets of outputs: the one from~\cite[Example 24]{BHM19}:
  \begin{equation}\label{eq:MAPK_out6}
      y_1 = K, \quad y_2 = F,\quad y_3 = S_{00},\quad y_4 = S_{01}, \quad y_5 = S_{10}, \quad y_6 = S_{11},
  \end{equation}
  a restricted version of~\eqref{eq:MAPK_out6}
  \begin{equation}\label{eq:MAPK_out5}
      y_1 = F,\quad y_2 = S_{00},\quad y_3 = S_{01}, \quad y_4 = S_{10},\quad y_5 = S_{11},
  \end{equation}
  and a ``symmetrized'' version of~\eqref{eq:MAPK_out6}
  \begin{equation}\label{eq:MAPK_out5bis}
      y_1 = K, \quad y_2 = F,\quad y_3 = S_{00},\quad y_4 = S_{01} + S_{10}, \quad y_5 = S_{11}.
  \end{equation}
\end{example}

\begin{example}[SEAIJRC model]\label{ex:seaijrc}
  The following extension of the SIR model was originally designed for studying infuenza, and its identifiability was studied by Roosa and Chowell~\cite[Model~2]{RC19} and Massonis et al~\cite[Model~41]{MBV20}:
  \begin{equation}\label{eq:SEAIJRC}
  \begin{cases}
      S' = -\beta S (I + J + q A) / N,\\
      E' = \beta S (I + J + q A) / N - k E,\\
      A' = k (1 - \rho) E - \gamma_1 A,\\
      I' = k \rho E - (\alpha + \gamma_1) I,\\
      J' = \alpha I - \gamma_2 J,\\
      C' = \alpha I,\\
      R' = \gamma_1(A + I) + \gamma_2 J,
  \end{cases}
  \end{equation}
  where $N = S + E + A + I + J + C + R$ is a known constant.
  Since $R$ does not appear in the right-hand sides and not present in the output (presented below), we  omit it.
  We use the observable $y_1 = C$ proposed in~\cite{MBV20} and add $y_2 = N$ to indicate that the total population is known.
\end{example}

\begin{example}[Goodwin oscillator]\label{ex:Goodwin}
The following model describes the oscillations in enzyme kinetics~\cite{Goodwin} (see also~\cite[Case study~1]{comparison}):
\begin{equation}\label{eq:Goodwin_orig}
  \begin{cases}
    \dot{x}_1 = -bx_1  + \frac{a}{A + x_3^\sigma},\\
    \dot{x}_2 = \alpha x_1 - \beta x_2,\\
    \dot{x}_3 = \gamma x_2 - \delta x_3,\\
    y_1 = x_1.
  \end{cases}
\end{equation}
The system is not polynomial but the following way of polynomializing without increasing the dimension of the parameter space has been suggested in~\cite[Supplementary materials~A.3]{SIAN}: we introduce a new parameter $c$ and a new state variable $x_4$ defined by  $c = \frac{A}{a}$ and  $x_4 = \frac{x_3^\sigma}{a}$, respectively.
Then the system~\eqref{eq:Goodwin_orig} becomes
\begin{equation}\label{eq:Goodwin}
\begin{cases}
    \dot{x}_1 = -bx_1  + \frac{1}{c + x_4},\\
    \dot{x}_2 = \alpha x_1 - \beta x_2,\\
    \dot{x}_3 = \gamma x_2 - \delta x_3,\\
    \dot{x}_4 = \sigma x_4 \frac{\gamma x_2 - \delta x_3}{x_3},\\
    y_1 = x_1.
\end{cases}
\end{equation}
\end{example}

\begin{example}[Akt pathway]\label{ex:Akt}
The following model for Akt pathway was developed by Fujita et al~\cite{Fujita2010}:
\begin{equation}\label{eq:Akt}
\begin{cases}
  x_1' = (k_{12} - k_{11}) x_9 + a(u - x_1),\\
  x_2' = k_{91} x_9 - k_{41} x_2 + k_{22} x_3 + k_{31} x_3 - k_{21} x_2 x_4,\\
  x_3' = k_{21} x_2 x_4 - k_{31} x_3 - k_{22} x_3,\\
  x_4' = k_{71} x_5 + k_{22} x_3 - k_{21} x_2 x_4,\\
  x_5' = k_{52} x_7 - k_{71} x_5 + k_{61} x_7 + k_{31} x_3 - k_{51} x_5 x_6,\\
  x_6' = k_{52} x_7 + k_{81} x_8 - k_{51} x_5 x_6,\\
  x_7' = k_{51} x_5 x_6 - k_{61} x_7 - k_{52}x_7,\\
  x_8' = k_{61} x_7 - k_{81}x_8,\\
  x_9' = k_{11} x_9 - k_{91} x_9 - k_{12} x_9,
\end{cases}
\end{equation}
where $u$ is an input variable.
The outputs used in~\cite{Fujita2010} were
\begin{equation}\label{eq:Akt_out}
\begin{cases}
  y_1 = a_1(x_2 + x_3),\\
  y_2 = a_2(x_5 + x_7),\\
  y_3 = a_3x_6.
\end{cases}
\end{equation}
\end{example}

\end{document}